\documentclass[11pt,a4paper]{article}

\usepackage[english]{babel}
\usepackage[blocks]{authblk}

\usepackage{amsmath}
\usepackage{amssymb}
\usepackage{enumitem}

\mathcode`l="8000
\begingroup
\makeatletter
\lccode`\~=`\l
\DeclareMathSymbol{\lsb@l}{\mathalpha}{letters}{`l}
\lowercase{\gdef~{\ifnum\the\mathgroup=\m@ne \ell \else \lsb@l \fi}}%
\endgroup

\makeatletter
\let\@fnsymbol\@arabic
\makeatother

\setlist[enumerate]{left=0pt,label=(\roman*), ref=(\roman*),font=\normalfont,topsep=-1ex,parsep=-.3ex,partopsep=0pt}

\usepackage{xcolor}
\usepackage{tikz}
\usetikzlibrary{calc}
\usetikzlibrary{decorations.pathreplacing}

\usepackage[style=math, giveninits=true,dashed=false]{biblatex}
\usepackage[thmmarks,hyperref]{ntheorem}
\theoremstyle{break}
\theoremheaderfont{\normalfont\bfseries}
\theorembodyfont{\slshape}
\theoremseparator{.}
\theorempreskip{3ex}
\theorempostskip{\topsep}
\theoremnumbering{arabic}
\theoremsymbol{}
\newtheorem{theorem}{Theorem}
\newtheorem{cl}{Claim}

\theoremstyle{plain}
\newtheorem{lemma}{Lemma}
\newtheorem{obs}{Observation}
\newtheorem{stat}{Statement}

\makeatletter
\newtheoremstyle{proof}%
{\item[\rlap{\vbox{\hbox{\hskip\labelsep \theorem@headerfont
##1\theorem@separator}\hbox{\strut}}}]}%
{\item[\rlap{\vbox{\hbox{\hskip\labelsep \theorem@headerfont ##1\ ##3\theorem@separator}\hbox{\strut}}}]}
\makeatother

\theoremheaderfont{\normalfont\bfseries}
\theorembodyfont{\upshape}

\theoremstyle{proof}

\theorempreskip{3.5ex}
\theorempostskip{\topsep}
\theoremseparator{.}

\theoremsymbol{\ensuremath{_\Box}}

\newtheorem{proof}{Proof}

\theoremheaderfont{\itshape}
\theoremsymbol{\ensuremath{_\Diamond}}

\newtheorem{smallproof}{Proof}

\newcommand{\specialedges}{E'}

\usepackage{hyperref}
\hypersetup{
    colorlinks=true,
    linkcolor=black,
    citecolor=black,
    pdftitle={Rooted Minors and Locally Spanning Subgraphs},
    pdfauthor={Böhme, Harant, Kriesell, Mohr, Schmidt},
    bookmarks=true,
    pdfpagemode=UseOutlines,
}

\title{Rooted Minors and Locally Spanning Subgraphs}

\newcommand\DAADmark{\footnotemark[1]}

\author{Thomas Böhme\thanks{Project initiated by partial supported from DAAD, Germany (as part of BMBF) and the Ministry of Education, Science, Research and Sport of the Slovak Republic within the project 57447800.}}
\author{Jochen Harant\protect\DAADmark}
\author{Matthias Kriesell\protect\DAADmark}
\affil{Institut für Mathematik der Technischen Universität Ilmenau, Weimarer Straße 25, 98693 Ilmenau, Germany}

\author{Samuel Mohr\protect\DAADmark\enspace\thanks{Supported by the MUNI Award in Science and
Humanities of the Grant Agency of Masaryk University.}\enspace\thanks{Previous affiliation:  Technische Universität Ilmenau.}\enspace\thanks{Gefördert durch die Deutsche Forschungsgemeinschaft (DFG) -- 327533333.}}
\affil{Masarykova univerzita Brno, Fakulta informatiky, \authorcr Botanická 554/68a, 602 00 Brno, Česká republika}

\author{Jens M. Schmidt\protect\DAADmark\enspace\thanks{This research is supported by the grant SCHM 3186/2-1 (401348462) from the Deutsche Forschungsgemeinschaft (DFG, German Research Foundation)}}
\affil{University of Rostock, Institute of Computer Science,
Albert-Einstein-Straße 22, 18059 Rostock, Germany}

\begin{document}
\maketitle

\begin{abstract}
\setlength{\parindent}{0em}
\setlength{\parskip}{1.5ex}
\noindent

\noindent 
Results on the existence of various types of spanning subgraphs of graphs are milestones in structural graph theory and have been diversified in several directions. In the present paper, we consider ``local'' versions of such statements.
In 1966, for instance,  D.\,W. Barnette proved that a $3$-connected planar graph contains a spanning tree of maximum degree at most $3$. A local translation  of this statement is  that if $G$ is a planar graph, $X$ is a subset of specified vertices of $G$ such that $X$ cannot be separated in $G$ by removing $2$ or fewer vertices of $G$, then $G$ has a  tree of maximum degree at most $3$ containing all vertices of $X$.

Our results constitute a general
machinery for strengthening statements about $k$-connected graphs (for $1
\leq k \leq 4$) to locally spanning versions, i.\,e.\ subgraphs containing a set $X\subseteq V(G)$ of a (not necessarily planar) graph $G$ in which
only $X$ has high connectedness.
Given a graph $G$ and $X\subseteq V(G)$, we say $M$ is a \emph{minor of $G$ rooted at $X$}, if $M$ is a minor of $G$ such that each bag of $M$ contains at most one vertex of $X$ and $X$ is a subset of the union of all bags.
We show that $G$ has a highly connected minor rooted at $X$ if $X\subseteq V(G)$ cannot be separated in $G$ by removing a few vertices of $G$. 

Combining these investigations and the theory of Tutte paths in the planar case yields to
locally spanning versions of six well-known results about degree-bounded trees, hamiltonian paths and cycles, and $2$-connected subgraphs of  graphs.

\medskip


\end{abstract}

\section{Introduction and $\boldsymbol{X}$\!-spanning Subgraph Results}\label{int}

In the present paper, we consider simple, finite, and undirected graphs; $V(G)$ and $E(G)$ denote the vertex set and the edge set of a graph $G$, respectively. For graph terminology not defined here, we refer to~\cite{diestel2017graph}. 
For a graph $G$ and a set  $X\subseteq V(G)$ of specified vertices, we say that a subgraph $H$ of $G$ is an \emph{$X$\!-spanning} subgraph of $G$ if $H$ contains $X$.  As usual, $H$ is  \emph{spanning} in the case of $X=V(G)$. For a positive integer $t$, a \emph{$t$-tree} is a tree with maximum degree at most $t$. 

This paper aims to constitute a general
machinery for strengthening statements about $k$-connected graphs to locally spanning versions. 
In particular, we translate well known results about spanning subgraphs to propositions about $X$\!-spanning subgraphs. 
As a starting point, we list six famous results.
Local versions of all these statements will be proved later; the forthcoming Theorems~\ref{f3} and~\ref{S} capture these results. 

To prove the above mentioned strengthening, we develop a new result on the existence of highly connected rooted minors (see Theorem~\ref{main} in Section~\ref{M}). 
Moreover, the proofs will draw tools from the theory of Tutte paths in $2$-connected plane graphs.

\bigskip

For $3$-connected planar graphs, Barnette, Biedl, and Gao proved the following Statements~\ref{Barnette} and~\ref{Gao}, where
Statement~\ref{Barnette} is best possible since there are $3$-connected planar graphs without a hamiltonian path.

\begin{stat}[D.~W.~Barnette~\cite{barnette1966trees}, T.~Biedl~\cite{biedl2014trees}]\label{Barnette}
If $G$ is a $3$-connected planar graph and $uv\in E(G)$, then $G$ has a spanning $3$-tree, such that  $u$ and $v$ are
leaves of that tree.
\end{stat}

\begin{stat}[Z.~Gao~\cite{gao19952}]\label{Gao}
A $3$-connected  planar graph $G$ contains a $2$-connected spanning subgraph of maximum
degree at most $6$.
\end{stat}

In~\cite{barnette19942}, it is shown that the constant $6$ in Statement~\ref{Gao} cannot be replaced with $5$.

Tutte~\cite{tutte1956theorem} proved that every $4$-connected planar graph has a hamiltonian cycle, and Thomassen~\cite{thomassen1983theorem} generalized this
result by showing that every $4$-connected
planar graph has a hamiltonian path connecting every given pair of vertices. Eventually, Sanders~\cite{sanders1997paths} extended the results of Thomassen and
of Tutte and proved the following statement.

\begin{stat}[D.~P.~Sanders~\cite{sanders1997paths}]\label{Sanders}
Every $4$-connected planar graph $G$ has a hamiltonian path between any two
specified vertices $u$ and $v$ and
containing any specified edge other than $uv$. 
\end{stat}

In \cite{goring2010hamiltonian}, it is shown that Statement \ref{Sanders} is best possible in the sense that there are $4$-connected maximal planar graphs with three edges
of large distance apart such that any hamiltonian
cycle misses one of them.

We know that $4$-connected planar graphs are hamiltonian, that means that they have a cycle through all vertices. 
An immediate consequence of Statement~\ref{Sanders} is that every $4$-connected planar graph even has a cycle containing all but one vertex. 
This raises the natural question whether those graphs have a cycle through all but two vertices. An affirmative answer was conjectured by Plummer~\cite{plummer1975problem}. Thomas and Yu gave a proof and showed Statement~\ref{ToYu}.

\begin{stat}[R.~Thomas, X.~Yu~\cite{thomas19944}]\label{ToYu}
A graph obtained from a $4$-con\-nec\-ted planar graph $G$ on at least $5$ vertices by deleting $2$
vertices  is hamiltonian.
\end{stat}

  Clearly, if three vertices of a
$4$-separator of a $4$-connected planar graph are removed, then the resulting graph does not contain a hamiltonian cycle, thus, Statement~\ref{ToYu} is
best possible.

For not necessarily planar graphs, Statements \ref{OtOz} and  \ref{Ellingham} hold.

\begin{stat}[K.~Ota, K.~Ozeki~\cite{ota2009spanning}]\label{OtOz}
Let $t \ge 4$ be an even integer and let $G$ be a $3$-connected graph.
If $G$ has no $K_{3,t}$-minor, then $G$ has a spanning $(t-1)$-tree.
\end{stat}

For a surface $\Sigma$, the \emph{Euler characteristic $\chi$} is defined by $\chi = 2-2g$ if $\Sigma$ is an
orientable surface of genus $g$, and by $\chi = 2-g$ if $\Sigma$ is a non-orientable surface of
genus $g$. Ellingham showed the following result.

\begin{stat}[M.~Ellingham~\cite{ellingham1996spanning},~\cite{ozeki2011spanning}]\label{Ellingham}
Let $G$ be a $4$-connected graph embedded on a surface of Euler characteristic $\chi <0$. Then $G$ has a spanning $\lceil \frac{10-\chi }{4}\rceil$-tree.
\end{stat}

In the sequel, \emph{$X$\!-spanning versions} of all six statements listed above are given in Theorem \ref{f3} and Theorem \ref{S}. 
Considering that, we define the connectedness of a set $X\subseteq V(G)$ in a graph $G$ first.
A set $S\subset V(G)$ is an \emph{$X$\!-separator} of $G$ if at least two components of $G-S$ obtained from $G$ by removing
$S$ contain a vertex of $X$.

Let $\kappa_G (X)$ be the maximum integer less than or  equal to $|X|-1$ such that the cardinality of each $X$\!-separator $S\subset V(G)$ --- if any exists --- is at least $\kappa_G (X)$. 
It follows that $\kappa_G (X)=|X|-1$  if $G[X]$ is complete, where  $G[X]$ denotes the subgraph of $G$ \emph{induced} by $X$; however, if $X$ is a proper subset of $V(G)$, then the converse need not be true.
If $\kappa_G(V(G))\ge k$ for a graph $G$, then we say that $G$ is \emph{$k$-connected}, and a $V(G)$-separator of $G$ is a \emph{separator} of $G$. 
This terminology corresponds to the commonly used definition of connectedness, e.\,g.\ in~\cite{diestel2017graph}.

In Theorem~\ref{f3}, local versions of Statements~\ref{Barnette}, \ref{Gao}, and~\ref{OtOz} are presented. 
The proof follows almost immediately from the statements
and Theorem~\ref{main} from Section~\ref*{M}. 
A detailed proof is given in Section~\ref{P12}. 

\begin{theorem} \label{f3}
\begin{enumerate}
\item If $G$ is a planar graph, $X\subseteq V(G)$, and $\kappa_G(X)\ge 3$,   then $G$ contains an $X$\!-spanning $3$-tree $T$. Moreover, if $x_1x_2\in E(G[X])$, then $T$ can be chosen such that $x_1$ and $x_2$ are leaves of $T$.
\item If $G$ is a planar graph, $X\subseteq V(G)$, and $\kappa_G(X)\ge 3$,   then $G$ contains
  a $2$-connected $X$\!-spanning subgraph $H$ of maximum degree at most~$6$.
\item If $t \ge 4$ is an even integer, $X\subseteq V(G)$ for a graph $G$, $\kappa_G(X)\ge 3$, and
$G$ has no $K_{3,t}$-minor, then $G$ has an $X$\!-spanning $(t-1)$-tree.
\end{enumerate}
\end{theorem}

Note that the local version of Barnette's result, which was stated in the abstract, follows in case $|X|\ge 4$ from Theorem~\ref{f3}~(i), whereas the case $|X|\le 3$ is trivial.

\begin{theorem}\label{S}
\begin{enumerate}
\item If $G$ is a planar graph, $X\subseteq V(G)$,  $\kappa_G (X)\ge 4$,
$x_1,x_2\in X$, $\specialedges\subseteq E(G[X])$, $|\specialedges|\le 1$, and $x_1x_2\notin \specialedges$, then $G$  contains an $X$\!-spanning path $P$ connecting $x_1$ and $x_2$ with $\specialedges\subseteq E(P)$.
\item If $G$ is a planar graph, $X\subseteq V(G)$, $\kappa_G (X)\ge 4$, and
$Y$ is a set of at most two vertices of $G$, then $G-Y$  contains an $(X\setminus Y)$-spanning cycle.
\item Let $G$ be a  graph embedded on a surface of Euler characteristic $\chi <0$, $X\subseteq V(G)$,  and $\kappa_G (X)\ge 4$. Then $G$ has an $X$\!-spanning $(\lceil \frac{10-\chi }{4}\rceil+1)$-tree.
\end{enumerate}
\end{theorem}

Theorem~\ref{S}~(i) extends Statement~\ref{Sanders} showing that there exists an $X$\!-spanning path connecting two vertices from $X$. 
We want to add here that 
Theorem~\ref{S}~(i) does not hold if $x_1\in V(G)\setminus X$: take a planar graph $H$ containing $X$ such that 
$\kappa_H(X)\ge 4$. Let $x_2\in X$ and let $G$ be obtained from $H$ by adding a pending path $P$ connecting $x_1$ and $x_2$.
Clearly, $\kappa_G(X)\ge 4$ but every path in $G$ connecting $x_1$ and $x_2$ is $P$. 

In contrast to the setting of Statement~\ref{Sanders}, Theorem~\ref{S}~(i) does not imply that there exists an $X$\!-spanning cycle. 
This case is covered by Theorem~\ref{S}~(ii) showing that planar graphs with $\kappa_G (X)\ge 4$ contain cycles through all vertices of $X$, all but one and all but two vertices. 
The set  $Y$ of at most two vertices of $G$ can be chosen arbitrarily (not necessarily from $X$). 
If $Y$ is a set of two vertices from $X$, this can be considered as an extension of Statement~\ref{ToYu}. 

Finally, Theorem~\ref{S}~(iii) is a local version of Statement~\ref{Ellingham}. 
We will see later in its proof that we use minors to build a base tree. Then, we will connect all missed vertices from $X$ to the base tree within the bags of the minor. 
That can lead in some cases to an increase of vertex degrees by 1, resulting in the ``$+1$'' in  Theorem~\ref{S}~(iii) in contrast to Statement~\ref{Ellingham}. 

\bigskip

The paper is organized as follows. In Section~\ref{M}, we introduce the concept of $X$\!-minors of $G$ for $X\subseteq V(G)$ and formulate Theorem~\ref{main} as the main result of the present paper. 
This statement is proved in Section~\ref{P3} and is used later as an auxiliary result for the proof of parts of Theorem~\ref{f3} and of Theorem~\ref{S}; however, Theorem~\ref{main} itself is an interesting contribution to the theory of rooted minors of graphs. 
In Section~\ref{P12}, the proofs of Theorems~\ref{f3} and~\ref{S} are presented by making use of Theorem~\ref{main} and the theory of Tutte paths in $2$-connected plane graphs.

\section{$\boldsymbol{X}$\!-Minors}\label{M}

Let $G$ be a graph and $\mathcal{M}$ be a family of pairwise disjoint subsets of
$V(G)$ such that these sets --- called \emph{bags} --- are non-empty and for each bag $A\subseteq V(G)$ the subgraph $G[A]$ induced by $A$ in $G$ is connected.
Let the bags of $\mathcal{M}$ be represented by the vertex set $V(M)$ of a graph $M$, then we say $\mathcal{M}=(V_v)_{v\in V(M)}$ is an \emph{$M$-certificate}
and $M$ is a \emph{minor} of $G$ if
there is an edge of $G$ connecting two bags $V_u$ and $V_v$ of $\mathcal{M}$ for every $uv\in E(M)$. 
As an equivalent definition (see~\cite{diestel2017graph}), a graph $M$ is a minor of a graph $G$ if it is isomorphic to a graph that can be obtained from a subgraph of $G$ by contracting edges.

In this section, we want to keep a set $X\subseteq V(G)$ of \emph{root vertices} alive in the minors. Therefore, we extend the concept of minors and introduce \emph{rooted minors}. \\
For  adjacent vertices $v, y\in V(G)$, let $G/vy$ denote the graph obtained from $G$ by removing $y$ and by adding a new edge $vz$
for every $z$ such that $yz\in E(G)$ and $vz\notin E(G)$.
That is, the edge $vy$ is \emph{contracted} into the vertex $v$ stated first (multiple edges do not occur); this is different from the standard notion of contraction, where a new
artificial vertex $z_{vy}$ is introduced as to replace both $v$ and $y$.
We call an edge $vy$ of $G$ \emph{$X$\!-legal} if $y\notin X$.
While this distinguishes $vy$ from $yv$, both notions refer to the same undirected edge.

A graph $M$ is a \emph{minor of $G$ rooted at $X$} or, shortly, an \emph{$X$\!-minor of $G$} if it can be obtained from a subgraph of $G$ containing $X$ by a (possibly empty) sequence of contractions of $X$\!-legal edges. 
Lemma~\ref{lem:Xminor} shows that there is an equivalent definition of a minor of $G$ rooted at $X$ by using certificates:

\begin{lemma}\label{lem:Xminor}
Let $G$ be a graph and $X\subseteq V(G)$. 
If $M$ is a graph with $X\subseteq V(M)$ and there is an $M$-certificate $\mathcal{M}=(V_v)_{v\in V(M)}$ of $G$, 
then $M$ is an $X$\!-minor of $G$ if and only if 
$v\in V_v$ for all $v\in V(M)$.
\end{lemma}

\begin{proof}[of Lemma~\ref{lem:Xminor}]
Suppose $M$ and $\mathcal{M}$ fulfil $v\in V_v$ for all $v\in V(M)$.
Then $G'=G[\bigcup_{v\in V(M)}V_v]$ is a subgraph of $G$. We obtain a subgraph $G''$ of $G'$ by removing all edges between $V_v$ and $V_w$ for all distinct $v,w\in V(M)$ with $vw\notin E(M)$. 
Starting with $G''$ and  repeatedly contracting $X$\!-legal edges $vy$  with $v\in V(M)$ and $y\in V_v\setminus\{v\}\subseteq V(G)\setminus X$ as long as there is $v\in V(M)$ with $|V_v|\geq 2$, we obtain $M$.
Hence, $M$ is an $X$\!-minor of $G$.\\
Now, let $M$ be an $X$\!-minor of $G$ obtained from a subgraph $G'$ of $G$ by contracting edges. 
We partition $V(G')$ by defining $V_v$  for every $v\in V(M)$. Let $V_v=\{v\}$ and iteratively add back all vertices $y\in V(G')$ to $V_v$ if $wy$ was contracted to $w\in V_v$. 
Then $\mathcal{M}=(V_v)_{v\in V(M)}$ is an $M$-certificate, $X\subseteq V(M)$, and $v\in V_v$ for $v\in V(M)$. 
\end{proof}

Note that an $\emptyset$-minor of $G$ is a minor of $G$ in the usual sense
whereas a minor of $G$ is isomorphic to some $\emptyset$-minor of $G$.
In this paper  the set $X$ is never empty.

If for an $X$\!-minor $M$ of $G$ there is an isomorphism $\varphi$ from a  subdivision of $M$ into a subgraph  of
$G$ such that all vertices of $M$ are fixed by $\varphi$, 
then $M$ is called a \emph{topological $X$\!-minor} of  $G$.

\bigskip

 In the remainder of this section, we deal with the question whether, for a given graph $G$ and $X\subseteq V(G)$, $G$ has a highly connected $X$\!-minor or even a highly connected topological $X$\!-minor if $\kappa_G(X)$ is large. 
An answer is given by  the forthcoming Theorem~\ref{main};
its proof can be found in Section~\ref{P3}.

\begin{theorem}\label{main}
Let $k\in \{1,2,3,4\}$, $G$ be a graph, and $X\subseteq V(G)$ such that $\kappa_G(X)\ge k$. Then: 
\begin{enumerate}
\item  $G$ has a $k$-connected $X$\!-minor.
\item  If $1\le k\le 3$, then $G$ has a $k$-connected topological $X$\!-minor.
\end{enumerate}
\end{theorem}

\bigskip

In the next three observations, we present examples showing that this theorem is best possible. 

\begin{obs}\label{thm1-obs1}
\text{\normalfont Theorem~\ref{main}~(i)} is best possible, because there are infinitely many (planar) graphs $G$ with the property that $G$ contains  $X\subseteq V(G)$ such that $\kappa_G(X)=6$ and $G$ has no $5$-connected $X$\!-minor.
\end{obs}

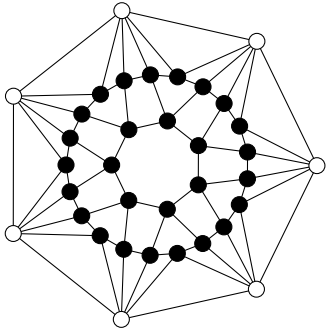
\begin{figure}[h]
\begin{center}
\begin{tikzpicture}[scale=.3][line width=42.5mm]

\foreach \i in {1,...,7}
{
\coordinate  (m\i) at (\i*51.4+25.7:2);
\coordinate  (a\i) at (\i*51.4-8.5:4);
\coordinate  (b\i) at (\i*51.4+8.5:4);
\coordinate  (c\i) at (\i*51.4+25.7:4);
\coordinate  (r\i) at (\i*51.4:7);

\coordinate (rp\i) at (51.4*\i-51.4:7);
\coordinate (cp\i) at (51.4*\i-51.4+25.7:4);
\coordinate  (mp\i) at (\i*51.4-25.7:2);

\draw (m\i) -- (mp\i);
\draw (mp\i) -- (a\i);
\draw (m\i) -- (b\i);

\draw (a\i) -- (b\i);
\draw (b\i) -- (c\i);
\draw (a\i) -- (cp\i);

\draw (a\i) -- (r\i);
\draw (b\i) -- (r\i);
\draw (c\i) -- (r\i);
\draw (r\i) -- (rp\i);
\draw (cp\i) -- (r\i);
}

\foreach \i in {1,...,7}
{
\node  (a\i) [circle, draw,fill=black, inner sep=0pt, minimum width=6pt ] at (\i*51.4-8.5:4) {};
\node  (b\i) [circle, draw,fill=black, inner sep=0pt, minimum width=6pt ] at (\i*51.4+8.5:4) {};
\node  (c\i) [circle, draw,fill=black, inner sep=0pt, minimum width=6pt ] at (\i*51.4+25.7:4) {};
\node  (m\i) [circle, draw,fill=black, inner sep=0pt, minimum width=6pt ] at(\i*51.4+25.7:2)  {};
\node  (r\i) [circle, draw,fill=white, inner sep=0pt, minimum width=6pt ] at (\i*51.4:7) {};
}
\end{tikzpicture}
\end{center}
\caption{The graph $G_7$. }\label{fig:g7}
\end{figure}

\begin{smallproof}
For an integer $t\ge 7$,  the graph $G_7$ of Figure~\ref{fig:g7} can  be readily generalized  to a plane graph $G_t$ containing a set $X$ of $t$ white vertices of degree $6$ forming a $t$-gon of $G_t$ and $4t$ black vertices of degree $4$ such that $\kappa_{G_t}(X)= 6$.
That $\kappa_{G_t}(X)= 6$ can be seen since any two nonadjacent
vertices of $X$ are connected by two subpaths of the outer cycle, two paths that use subpaths
of the middle cycle, and two paths that cross the middle cycle and use subpaths of the inner
cycle. 
The assertion is proved, if there is no $5$-connected $X$\!-minor  $M$ of $G_t$.

 Assume that $M$ exists and that $M$ is obtained  from a subgraph $H$ of $G_t$ by contractions of $X$\!-legal edges. If $|V(G_t)\setminus V(H)|=b$, then we can say that $M$ is obtained from $G_t$ by a number $a$ of contractions of $X$\!-legal edges and by $b$  removals of vertices not belonging to $X$. If an $X$\!-legal edge $vy$ is contracted or a vertex $z\notin X$  is removed, then the degree of a vertex distinct from $v,y$ or distinct from $z$, respectively, does not increase. Since $G_t$ has $4t$ vertices of degree $4$ and the minimum degree $\delta(M)$ of $M$ is at least $5$, each black vertex either must be  removed or an incident edge must be contracted. Thus, it follows $2a+b\ge 4t$ implying $a+b\ge 2t$. Because $n=|V(M)|=|V(G_t)|-(a+b)=5t-(a+b)$, we obtain $n\le 3t$. 

 Note that $M$, as an $X$\!-minor of a planar graph, is planar. Since $M$ is $5$-connected, it has, up to the choice of the outer face, a unique embedding into the plane. It is clear (consider the drawing of $G_7$ in Figure~\ref{fig:g7}) that
 the vertices of $X$ remain boundary vertices of a $t$-gon $\alpha$ of  such an embedding of $M$ into the plane.
For a vertex $x\in X$, let $N_M(x)$ be the set  of neighbors of $x$ in $M$, $|N_M(x)|\ge 5$. Furthermore,  $|N^*(x)|\ge 3$ for  $N^*(x)=N_M(x)\setminus X$ and $x\in X$, because otherwise the boundary cycle of $\alpha$ has a chord incident with $x$ and the end vertices of this chord form a separator of $M$,
contradicting the  $3$-connectedness, and therefore also the $5$-connectedness of $M$. If $N^*(x_1)\cap N^*(x_2)\neq \emptyset$ for non-adjacent $x_1,x_2\in X$, then $S=\{x_1,x_2,u\}$ with $u\in N^*(x_1)\cap N^*(x_2)$ is a separator of $M$, a contradiction. For the same reason $|N^*(x_1)\cap N^*(x_2)|\le 1$ for adjacent $x_1,x_2\in X$, and if $N^*(x_1)\cap N^*(x_2)=\{u\}$, then $x_1,x_2,$ and $u$ are the boundary vertices of a $3$-gon of $M$. It follows 
$$n=|V(M)|\ge |X|+|\bigcup_{x\in X}N^*(x)|\ge t+\sum_{x\in X}(|N^*(x)|-1)\ge 3t.$$
All together, $n=3t$, $V(M)= X\cup \allowbreak\bigcup_{x\in X}N^*(x)$, $|N^*(x)|=3$ for $x\in X$, $|N^*(x_1)\cap N^*(x_2)|= 0$ for non-adjacent $x_1,x_2\in X$, $|N^*(x_1)\cap N^*(x_2)|= 1$ for adjacent $x_1,x_2\in X$, and if $N^*(x_1)\cap N^*(x_2)=\{u\}$ in this case, then $x_1,x_2,$ and $u$ are the boundary vertices of a $3$-gon of $M$.\\
For $v\in \bigcup_{x\in X}N^*(x)$, it holds $|N_M(v)\cap \allowbreak X|\le 2$, thus, $|N_M(v)\cap \allowbreak (V(M)\setminus X)|=|N_M(v)\cap \allowbreak (\bigcup_{x\in X}N^*(x))|\ge 3$ and it is checked readily  that $v$ has a neighbor $w\in N^*(x')$ such that $x\neq x'$ and $\{x,x',v,w\}$ is a separator of $M$, a contradiction to the $5$-connectedness of $M$.  \end{smallproof}

\begin{obs}\label{thm1-obs2}
\text{\normalfont Theorem~\ref{main}~(ii)} is best possible, because for an arbitrary integer $l$, there is a (planar) graph $G$ and $X\subseteq V(G)$ with $\kappa_G(X)\ge l$ such that
every  topological $X$\!-minor of $G$ is not $4$-connected.
\end{obs}

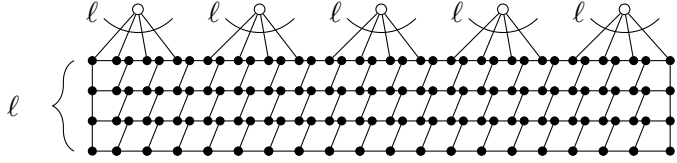
\begin{figure}[h]
\begin{center}
\begin{tikzpicture}[scale=.4,y=-1cm][line width=42.5mm]

\node (anker) at (-2.7,-2) {};

\foreach \y in {0,...,3}
{
\node  (0P\y a) [circle, draw,fill=black, inner sep=0pt, minimum width=3pt ] at (0,\y) {};
\node  (19P\y a) [circle, draw,fill=black, inner sep=0pt, minimum width=3pt ] at (19,\y) {};
}
\draw (0P0a) -- ($(0P3a)$);
\draw (19P0a) -- ($(19P3a)$);

\foreach \x in {1,...,18}
{
\foreach \y in {0,...,2}
{
\node  (\x P\y a) [circle, draw,fill=black, inner sep=0pt, minimum width=3pt ] at ( $(\x,\y)+(-.2,0)$) {};
\node  (\x P\y b) [circle, draw,fill=black, inner sep=0pt, minimum width=3pt ] at ( $(\x,\y)+(.2,0)$) {};
}
\node  (\x P3a) [circle, draw,fill=black, inner sep=0pt, minimum width=3pt ] at  ( $(\x,3)+(-.2,0)$) {};
\node  (\x P3b) [circle, draw,fill=black, inner sep=0pt, minimum width=3pt ] at  (\x P3a) {};

\foreach \y [evaluate=\y as \yn using int(\y+1)] in {0,...,2}
{
\draw ($(\x P\yn a)$) -- (\x P\y b);
}
}

\foreach \y in {0,...,2}
{
\draw (0P\y a) -- ($(1P\y a)$);
\draw (1P\y a) -- ($(18P\y b)$);
\draw (18P\y b) -- ($(19P\y a)$);
}
\draw (0P3a) --  ($(1P3a)$);
\draw (1P3a) -- ($(18P3b)$);
\draw (18P3b) --  ($(19P3a)$);

\foreach \i in {0,4,8,12,16} {
\node  (v\i) [circle, draw,fill=white, inner sep=0pt, minimum width=4pt ] at ($(\i,0)+(1.5,-1.7)$) {};
}

\foreach \x in {0,...,3}
{
\draw (\x P0a) -- (v0);
}

\foreach \x in {4,...,7}
{
\draw (\x P0a) -- (v4);
}
\foreach \x in {8,...,11}
{
\draw (\x P0a) -- (v8);
}
\foreach \x in {12,...,15}
{
\draw (\x P0a) -- (v12);
}

\foreach \x in {16,...,19}
{
\draw (\x P0a) -- (v16);
}

\foreach \i in {0,4,8,12,16} {
\draw ($(v\i)+(1.14,.3)$) arc (45:135:1.6);
\node (t) at ($(v\i)+(-1.5,.1)$) {$l$};
}

\draw [decorate,decoration={brace,mirror,amplitude=8pt}]
 ($(0P0a)+(-.6,0)$)  --  ($(0P3a)+(-.6,0)$) node [black,midway,xshift=-.8cm] {$l$};

\end{tikzpicture}
\end{center}
\caption{The graph $F_l$ (number of white vertices is larger than $l$). }\label{fig:Fl}
\end{figure}

\begin{smallproof}
For $l\ge 4$ consider the graph $F_l$ of Figure~\ref{fig:Fl} and let $X$ be the set of white vertices of $F_l$ with $|X|\geq l+1$.
The vertices of $X$ have degree $l\ge 4$ and all black vertices have degree at most $3$ in $F_l$. Moreover, it is easy to see that
$\kappa_{F_l}(X)= l$. Suppose, to the contrary, that there is  a $4$-connected topological $X$\!-minor $M$ of $F_l$ and an isomorphism $\varphi$ from a subdivision of $M$ into a subgraph $H$ of $F_l$. 
Then each vertex $v\in V(M)$ is a vertex of $H$ and has degree at least $4$ in $H$ and, therefore, also in $F_l$, thus,
$v\in X$. Since $X\subseteq V(M)$ it follows $X=V(M)$. The vertices of $X$ are boundary vertices of a common face in $F_l$, hence, also in $M$.
Consequently, $M$ is a simple outerplanar graph implying $\delta(M)=2$, a contradicting $\delta(M)\ge 4$.
\end{smallproof}

This shows also, that there cannot be any integer $l$ such that $\kappa_G(X)\ge l$ implies the existence of a $4$-connected topological $X$\!-minor. 

By the first example,  it remains open whether an integer $l$ exists --- it must be at least $7$ --- such that every graph $G$ containing  $X\subseteq V(G)$ with $\kappa_G(X)\ge l$ has a $5$-connected $X$\!-minor.
We  conclude this section by showing:

\begin{obs}\label{thm1-obs3}
There cannot be any integer $l$ such that $\kappa_G(X)\ge l$ implies the existence of a $6$-connected $X$\!-minor.
\end{obs}

\begin{figure}[h]
\begin{center}
\begin{tikzpicture}[scale=.4,y=-1cm][line width=42.5mm]

\foreach \x in {0,...,15}
{
\foreach \y in {0,...,5}
{
\node  (\x P\y) [circle, draw,fill=black, inner sep=0pt, minimum width=3pt ] at (\x,\y) {};
}
\draw (\x P0) -- ($(\x P5)+(0,.3)$);
\draw[dotted] ($(\x P5)+(0,.5)$) -- ($(\x P5)+(0,1)$);

\node  (\x P6) [circle, draw,fill=black, inner sep=0pt, minimum width=3pt ] at ($(\x P5)+(0,1.6)$) {};
\node  (\x P7) [circle, draw,fill=black, inner sep=0pt, minimum width=3pt ] at ($(\x P5)+(0,2.6)$) {};
\draw ($(\x P6)+(0,-.3)$) -- (\x P7);
}

\foreach \y in {0,...,7}
{
\draw (0P\y) -- ($(15P\y)+(.5,0)$);
\draw[dotted] ($(15P\y)+(1,0)$) -- ($(15P\y)+(1.7,0)$);

\node  (16P\y) [circle, draw,fill=black, inner sep=0pt, minimum width=3pt ] at ($(15P\y)+(2.7,0)$) {};
\node  (17P\y) [circle, draw,fill=black, inner sep=0pt, minimum width=3pt ] at ($(15P\y)+(3.7,0)$) {};
\node  (18P\y) [circle, draw,fill=black, inner sep=0pt, minimum width=3pt ] at ($(15P\y)+(4.7,0)$) {};
\draw ($(16P\y)+(-.5,0)$) -- (18P\y);
}

\foreach \x in {16,...,18}
{
\draw (\x P0) -- ($(\x P5)+(0,.3)$);
\draw[dotted] ($(\x P5)+(0,.5)$) -- ($(\x P5)+(0,1)$);

\node  (\x P6) [circle, draw,fill=black, inner sep=0pt, minimum width=3pt ] at ($(\x P5)+(0,1.6)$) {};
\node  (\x P7) [circle, draw,fill=black, inner sep=0pt, minimum width=3pt ] at ($(\x P5)+(0,2.6)$) {};
\draw ($(\x P6)+(0,-.3)$) -- (\x P7);
}

\node  (v0) [circle, draw,fill=white, inner sep=0pt, minimum width=4pt ] at ($(0P0)+(2.5,-1.7)$) {};
\node  (v1) [circle, draw,fill=white, inner sep=0pt, minimum width=4pt ] at ($(6P0)+(2.5,-1.7)$) {};
\node  (v2) [circle, draw,fill=white, inner sep=0pt, minimum width=4pt ] at ($(12P0)+(2.5,-1.7)$) {};
\node  (vk) [circle, draw,fill=white, inner sep=0pt, minimum width=4pt ] at ($(18P0)+(-2.5,-1.7)$) {};

\foreach \x in {0,...,5}
{
\draw (\x P0) -- (v0);
}
\foreach \x in {6,...,11}
{
\draw (\x P0) -- (v1);
}
\foreach \x in {12,...,15}
{
\draw (\x P0) -- (v2);
}
\foreach \x in {16,...,18}
{
\draw (\x P0) -- (vk);
}
\draw[dotted,black] ($(v2)+(.8,0)$) -- ($(vk)+(-.8,0)$);

\foreach \i in {0,...,1}
{
\draw ($(v\i)+(1.7,.3)$) arc (45:135:2.4);
\node (t) at ($(v\i)+(-2.1,.1)$) {$l$};
}
\draw ($(v2)+(.65,.95)$) arc (75:135:2.4);
\node (t) at ($(v2)+(-2.1,.1)$) {$l$};
\draw ($(vk)+(1.7,.3)$) arc (45:105:2.4);
\node (t) at ($(vk)+(2.1,.1)$) {$l$};

\draw [decorate,decoration={brace,mirror,amplitude=12pt}]
 ($(0P0)+(-.6,0)$)  --  ($(0P7)+(-.6,0)$) node [black,midway,xshift=-.8cm] {$l$};

\draw [decorate,decoration={brace,amplitude=12pt}]
 ($(0P0)+(0,-2.5)$)  --  ($(18P0)+(0,-2.5)$) node [black,midway,yshift=0.8cm] {$l\cdot (l+1)$};
\end{tikzpicture}
\end{center}
\caption{The graph $H_l$. }\label{fig:Hl}
\end{figure}
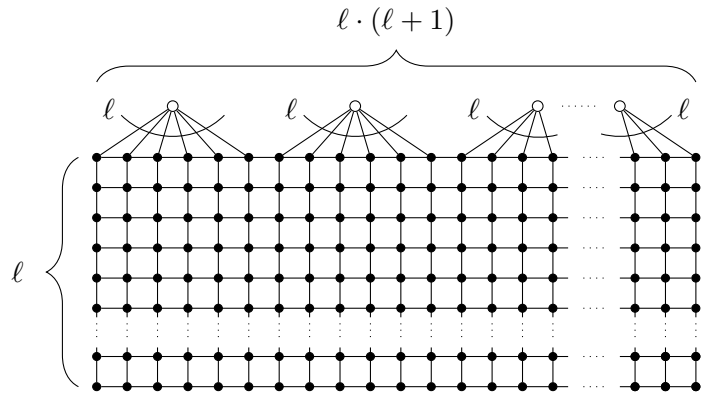

\begin{smallproof}
 Let $l\ge 6$ and consider the planar graph $H_l$ of Figure~\ref{fig:Hl}. It contains a set $X$ of $l+1$ white vertices of degree $l$
and further $l^2(l+1)$ black vertices. It is easy to see that $\kappa_{H_l}(X)=l$. An arbitrary $X$\!-minor $M$ of $H_l$ is also planar and, since
it contains $X$, it has at least $l+1\ge 7$ vertices. 
It is known that planar graphs are not $6$-connected. 
\end{smallproof}

\section{Proof of Theorem~\ref{main}}\label{P3}

In this section, we prove Theorem~\ref{main}. 
The following Lemma~\ref{Lemma1} --- as a consequence of Menger's Theorem~\cite{bohme2001menger,menger1927allgemeinen} --- and Lemma~\ref{Lemma2} will be used several times.

\begin{lemma}\label{Lemma1} Let $G$ be a graph, $X\subseteq V(G)$, $k\geq 1$, and $|X|\geq k+1$.\\
Then $\kappa_G (X)\ge k$ if and only if for every $x,y\in X$ with $xy\notin E(G)$ there are $k$ internally
vertex disjoint paths connecting $x$ and $y$. \end{lemma}

Let $S$ be an $X$\!-separator of $G$, the union $F$ of the vertex sets of at least one but not of all components of $G-S$ is called an \emph{$S$-$X$\!-fragment}, if both $F$ and $\overline{F}:=V(G-S)\setminus F$ contain at least one vertex from $X$. 
In this case,  $\overline{F}$ is an $S$-$X$\!-fragment, too.\\
For an $S$-$V(G)$-fragment $F$, we again drop the $V(G)$ in the notion; thus, $F$ is an \emph{$S$-fragment} for a separator $S$ of $G$. 
We say that  some set $Y\subseteq V(G)$ is \emph{$X$\!-free} if $Y\cap X=\emptyset$.

\begin{lemma}\label{Lemma2} Let $G$ be a  graph, $S\subset V(G)$ be a separator of $G$, and $F$ be an $X$\!-free $S$-fragment of $G$. Furthermore, let $G'$ be the graph obtained from  $G[\overline{F}\cup S]$ by adding all possible edges between vertices of $S$ (if not already present).\\
Then  $\kappa_{G'}(X)\ge \kappa_G(X)$. \end{lemma}

\begin{proof}[of Lemma~\ref{Lemma2}]
If  $G'[X]$ is complete, then $\kappa_{G'}(X)=|X|-1\ge \kappa_G(X)$, hence, Lemma~\ref{Lemma2} holds in this case. \\
Consider $x_1,x_2\in X$ such that $x_1$ and $x_2$ are  non-adjacent  in $G'$. Since $S$ forms a clique in $G'$, we may assume that $x_2\notin S$ (possibly $x_1\in S$).
According to Lemma~\ref{Lemma1}, we have to show that there are at least $\kappa_G(X)$ internally vertex disjoint paths in $G'$ connecting $x_1$ and $x_2$. Note that $x_1$ and $x_2$ are also non-adjacent in $G$ and, again using Lemma~\ref{Lemma1}, consider a set $\mathcal{P}$ of $\kappa_G(X)$ internally vertex disjoint paths of $G$ connecting $x_1$ and $x_2$.\\
If some $P\in\mathcal{P}$ is not a path of $G'$, then $P$ contains at least one subpath $Q$ on at least $3$ vertices connecting two vertices $u,v\in S$  such that
$V(Q)\cap V(G')=\{u,v\}$. We obtain a path connecting $x_1$ and $x_2$ from $P$ by removing all inner vertices of $Q$ and adding the edge $uv$. 
Note that $uv\in E(G')$ and repeating this procedure finally leads to a path $P'$ of $G'$. If $P\in\mathcal{P}$ is a path of $G'$, we put $P'=P$. \\
Since $V(P')\subseteq V(P)$ for all $P\in\mathcal{P}$ , the set $\mathcal{P}'=\{P'\mid P\in\mathcal{P}\}$ is a set of $\kappa_G(X)$ internally vertex disjoint paths  connecting $x_1$ and $x_2$. Since $x_1$ and $x_2$ have been chosen arbitrarily, Lemma~\ref{Lemma2} is proved.
\end{proof}

 First we prove Theorem~\ref{main}~(ii).

\begin{proof}[of Theorem~\ref{main}~(ii)]
Since $X$ is connected in $G$, there is a component $K$ of $G$ containing all vertices from $X$. 
If $K$ is  $k$-connected, then $K$ itself is a  $k$-connected topological $X$\!-minor of $G$ and (ii) is proved in this case.

Assume that (ii) is not true and let $G$ be a counterexample with the smallest number of vertices.
Then $G$ is connected and consider a smallest separator $S$ of $G$,  $|S|\le k-1\le 2$.
Since $\kappa_G(X)\ge k$,  there is  an $X$\!-free $S$-fragment $F$ of $G$ and $X\subseteq \overline{F}\cup S$.

Let $G'$ be obtained from $G[\overline{F}\cup S]$ by adding all possible edges between vertices of $S$ (if not already present), then, by Lemma~\ref{Lemma2},  $\kappa_{G'}(X)\ge k$.\\
 Since $G'$ has less vertices than $G$, $G'$ contains a subgraph $H'$ isomorphic to a subdivision  of a
$k$-connected $X$\!-minor $M'$ of $G'$. 
Note that $M'$ is also an $X$\!-minor of $G$, since we can contract $F$ into one of the at most two vertices of $S$ by performing only $X$\!-legal edge contractions.

If $H'$ is also a subgraph of $G$, then this contradicts the choice of $G$.
Thus, $k=3$, $\kappa_{G}(V(G))=2$, $S=\{u,v\}$ and $uv\in E(H')\setminus E(G)$.
In this case, let $H$ be obtained from $H'$ by replacing $uv$ with a path $Q$ of $G$ connecting $u$ and $v$ such that $V(Q)\cap
\overline{F}=\emptyset$. Then $H$ is a subgraph of $G$ and also isomorphic to a subdivision of $M'$, again a contradiction, and (ii) is proved.
\end{proof}

To prove Theorem~\ref{main}~(i), we show the following lemma first. This will enable us to find the desired minor in an iterative way.

\begin{lemma}\label{lem:proofmainthm}
Let $G$ be a connected graph and $X\subseteq V(G)$.
If $\kappa_G(X)\geq 4$, then there exists an $X$\!-legal edge $vy$ such that $\kappa_{G/vy}(X)\geq 4$, unless $G$ is $4$-connected. 
\end{lemma}

We start with proving the following claim. 

\begin{cl}\label{claim:kappax}
If $vy$ is an $X$\!-legal edge of a graph $G$ with $X\subseteq V(G)$, then $\kappa_{G/vy}(X)\ge \kappa_{G}(X)$ or $\kappa_{G/vy}(X)=\kappa_{G}(X)-1$ and the latter case holds if and only if there exists an $X$\!-separator  of $G$ of  size $\kappa_{G}(X)$ containing $v$ and $y$.
\end{cl}

\begin{smallproof}[of Claim~\ref{claim:kappax}]
We assume $\kappa_{G/vy}(X)<\kappa_{G}(X)$. Then $(G/vy)[X]$ is not complete, because otherwise
$|X|-1=\kappa_{G/vy}(X)<\kappa_{G}(X)$, contradicting $\kappa_{G}(X)\le |X|-1$.\\
Let $x_1,x_2\in X$ and $S\subset V(G/vy)$ be chosen such that  $|S|=\kappa_{G/vy}(X)$ and $S$ separates $x_1$ and $x_2$ in $G/vy$.\\
Then $x_1x_2\notin E(G/vy)$ and it follows $x_1x_2\notin E(G)$ because an edge in $E(G)\setminus E(G/vy)$ is incident with $y$. Since $|S|=\kappa_{G/vy}(X)<\kappa_{G}(X)$, $G-S$ contains a path $P$ connecting $x_1$ and $x_2$. If $y\notin V(P)$, then $P$ is also a path of $G/vy-S$, contradicting the choice of $S$. 
If $y\in V(P)$ and  $v\notin S$, then $v\in V(G/vy)$, $N_G(y)\setminus\{v\}\subseteq N_{G/vy}(v)$ and,
in both cases $v\in V(P)$ and $v\notin  V(P)$, it is easy to see that  $(G/vy)-S$ still contains a path connecting $x_1$ and $x_2$, again a contradiction.\\
All together,  $v\in S$ and every path of $G-S$ connecting $x_1$ and $x_2$ contains $y$. It follows that $S\cup \{y\}$ separates $x_1$ and $x_2$ in $G$, hence, $\kappa_{G}(X)\le |S\cup \{y\}|=\kappa_{G/vy}(X)+1\le \kappa_{G}(X)$.

If $vy$ is an $X$\!-legal edge of $G$ and there exists an $X$\!-separator  of $G$ of  size $\kappa_{G}(X)$ containing $v$ and $y$, 
then let $x_1,x_2\in X$ be chosen such that $S$ separates $x_1$ and $x_2$ in $G$. 
Each path that connects $x_1$ and $x_2$ in $G$ contains at least one vertex from $S$ and, therefore, every path that connects $x_1$ and $x_2$ in $G/vy$ contains at least one vertex from $S\setminus\{y\}$. It follows that $S\setminus\{y\}$ is an $X$\!-separator  of $G/vy$ and $\kappa_{G/vy}(X)\leq |S\setminus\{y\} |=\kappa_{G}(X)-1$. 
By the first statement of the claim, we get $\kappa_{G/vy}(X)=\kappa_{G}(X)-1$. 
\end{smallproof}

\begin{proof}[of Lemma~\ref{lem:proofmainthm}]

 Suppose that  $\kappa_G(X)\geq 4$ and $G$ is not $4$-connected. Since $|V (G)| \geq |X| > 4$, there must exist a separator $T$ with  $|T |\in\{1,2,3\}$.
Since at most one
component of $G-T$ contains vertices from $X$, there exists an $X$\!-free $T$\!-fragment $F$.
Let $t\in T$ and $y\in N_G(t)\cap F$, then $ty$ is $X$\!-legal, and it turns out by  Claim~\ref{claim:kappax} that $\kappa_{G/ty}(X)\geq 4$ if $G[X]$ is complete or if $\kappa_G(X)\geq 5$. 
Assume that $|T|\in \{1,2\}$. For all $x_1,x_2\in X$ with $x_1x_2\notin E(G)$, there are four internally vertex disjoint paths in $G$ connecting $x_1$ and $x_2$. 
If one of these paths contains $y$, then this path, say $P$, also contains $t$ and there is a path in $G/ty$ connecting $x_1$ and $x_2$ using only vertices from $P$; hence, $\kappa_{G/ty}(X)\geq 4$. \\
Therefore, if we assume that the statement of Lemma~\ref{lem:proofmainthm} does not hold, i.\,e., $\kappa_{G/vy}(X)< 4$ for every $X$\!-legal edge $vy$ of $G$, then $\kappa_G(V(G))=3$ and $\kappa_G(X)= 4$. 
Moreover, for every $X$\!-legal edge $vy$ of $G$ there is  minimum $X$\!-separator $S$ with $v, y \in S$.
For the remainder of the proof, we assume that every considered separator $T$ is minimum,  i.\,e.\ $|T|=3$, and for every $X$\!-legal edge $vy$ of $G$ there is an $X$\!-separator $S$ of $G$ with $v, y \in S$ and $|S|=4$.

\begin{cl}\label{lem:T}
Let  $S$ and $S'$ be  separators in a graph $G$. 
For an $S$\!-fragment $F$ and an $S'$\!-fragment $F'$,
let $T(F,F'):=(F\cap S')\cup(S'\cap S)\cup(S\cap F')$.\\
Then
\begin{enumerate}
\item  If $F\cap F'\neq \emptyset$, then $T(F,F')$ is a separator of $G$ separating $F\cap F'$ from the remaining graph, 
\item  $|T(F,F')|+|T(\overline{F},\overline{F'})|=|S|+|S'|$.
\end{enumerate}
\end{cl}

\begin{smallproof}[of Claim~\ref{lem:T}]
Since $S$ and $S'$ are separators, 
$N_G(F\cap F')\subseteq S\cup F$ and $N_G(F\cap F')\subseteq S'\cup F'$. Hence, $N_G(F\cap F')\subseteq T(F,F')$.
Since $\overline{F}\cup\overline{F'}\neq\emptyset$, $V(G)\neq T(F,F')\cup (F\cap F')$, thus $N_G(F\cap F')$ is a separator of $G$; and so is $T(F,F')$.
This proves (i) and easy counting leads to~(ii). 
\end{smallproof}

Now, let us go back to the situation that there is a  separator $T$ of $G$ with $|T|=\kappa_G(V(G))=3$.
We want to show that there is no edge $xy$ in $G$ with $x,y\in T$; that is that $T$ is an anticlique. 
Recall that an \emph{anticlique} of $G$ is a set $A$ of vertices of $G$ such that  $G[A]$ is an edgeless graph. 

\begin{cl}\label{claim:Tanti}
Let $T$ be a separator of $G$ with $|T|=3$. Then 
\begin{enumerate}
\item $T$ is an anticlique, 
\item if $F$ is an $X$\!-free $T$\!-fragment, $t\in T$, $y\in F\cap N_G(t)$, $S$ is an $X$\!-separator with $t,y\in S$ with $|S|=4$, and $B$ is an $S$-$X$\!-fragment, then $|B\cap T|=1$. Moreover, such an $X$\!-separator with $t,y\in S$ always exists.
\end{enumerate}
\end{cl}

\begin{smallproof}[of Claim~\ref{claim:Tanti}]
Let  $F$ be an $X$\!-free $T$\!-fragment and $t\in T$.
For $y\in F\cap N_G(t)$, the edge $ty$ is $X$\!-legal. 
Let $S$ be an $X$\!-separator with $t,y\in S$ and $|S|=4$. Its existence is ensured by Claim~\ref{claim:kappax}. 
Let $B$ be an $S$-$X$\!-fragment.
If $T\cap B=\emptyset$, then $B\cap \overline{F}$ is not $X$\!-free and is separated by $T(B,\overline{F})$ from $\overline{B}$ (Claim~\ref{lem:T}~(i)). 
But $T(B,\overline{F})=(B\cap T)\cup(T\cap S)\cup(S\cap \overline{F})\subseteq S\setminus\{y\}$ has at most three vertices, a contradiction to $\kappa_G(X)= 4$.\\
In the same vein, $T\cap \overline{B}\neq \emptyset$ and, because $|T|=3$, it follows $|B\cap T|=|\overline{B}\cap T|=1$ and the two vertices in $T\setminus\{t\}$ are non-adjacent.
Since $t$ has been chosen arbitrarily from $T$, $T$ is an anticlique in $G$. 
\end{smallproof}

\begin{cl}\label{claim:SFeins}
Let $T$ be a separator of $G$ with $|T|=3$ and $F$ be an $X$\!-free $T$\!-fragment, then $|F|=1$. 
\end{cl}

\begin{smallproof}[of Claim~\ref{claim:SFeins}]
Let $t\in T$, $y\in F\cap N_G(t)$, $S$ be an $X$\!-separator with $t,y\in S$ with $|S|=4$ (by Claim~\ref{claim:kappax}), and $B$ be an $S$-$X$\!-fragment.
If $|F\cap S|\geq 2$, then $|T(B,\overline{F})|\leq 3$ and $|T(\overline{B},\overline{F})|\leq 3$, and both $B\cap\overline{F}$ and $\overline{B}\cap\overline{F}$ are $X$\!-free, so that $X\subseteq T\cup (\overline{F}\cap S)$, contradicting $|X|\geq 5$. 
Hence $F\cap S=\{y\}$.
Let $t'$ be the unique vertex in $B\cap T$ by Claim~\ref{claim:Tanti}~(ii). 
\\
It follows that $B\cap F=\emptyset$ for otherwise this set would be an $\{t, y, t'\}$-fragment as $T(B,F)=\{t, y, t'\}$ is a separator of $G$ by Claim~\ref{lem:T}~(i); but $\{t, y, t' \}$ is not an anticlique since $ty\in E(G)$, which is  a contradiction to Claim~\ref{claim:Tanti}~(i).
Likewise, $\overline{B}\cap F =\emptyset$, so that $F = \{y\}$, and again, this holds for
every $X$\!-free $T$\!-fragment. 
\end{smallproof}

Now, let $T=\{t,t',t''\}$ be a separator of $G$, $F=\{y\}$ be an $X$\!-free $T$\!-fragment (Claim~\ref{claim:SFeins}), and $S$ be an $X$\!-separator with $t,y\in S$ and $|S|=4$. 
Then there is an $S$-$X$\!-fragment $B$ and unique vertices $t'$ and $t''$ in $B\cap T$ and $\overline{B}\cap T$, respectively (by Claim~\ref{claim:Tanti}~(ii)). 
There exists an $X$\!-separator $S'$ with $t',y\in S'$ and $|S'|=4$ by Claim~\ref{claim:kappax} and we may take an $S'$\!-$X$\!-fragment $B'$ such that $t\in B'$ and $t''\in \overline{B'}$ (by Claim~\ref{claim:Tanti}~(ii)).
This situation is sketched in Figure~\ref{fig:TSS}.

\begin{figure}[h]
\begin{center}
\begin{tikzpicture}[scale=.3][line width=42.5mm]
\filldraw[fill=black!50!, fill opacity=.2,rounded corners=8pt, line width=.7pt, draw=black] (-6.0,1)--(-6.0,-1)--(9,-1)--(9,1)--cycle;
\filldraw[fill=black!50!, fill opacity=.2,rounded corners=8pt, line width=.7pt, draw=black] (-1,6)--(-1,-9)--(1,-9)--(1,6)--cycle;

\node  (y) [circle, draw=black,fill=black, inner sep=0pt, minimum width=6pt,label={[label distance=-.16cm]70:$y$} ] at (-.2,-.2) {};
\node  (x) [circle, draw=black,fill=white, inner sep=0pt, minimum width=7pt,label={[label distance=.02cm]0:$t$} ] at (-4,0) {};
\node  (xp) [circle, draw=black,fill=white, inner sep=0pt, minimum width=7pt,label={[label distance=.02cm]-90:$t'$} ] at (0,4) {};
\node  (xpp) [circle, draw=black,fill=white, inner sep=0pt, minimum width=7pt,label={[label distance=-.08cm]45:$t''$} ] at (5,-5) {};

\node  (Sp) at (0,7.3) {$S'$};
\draw [decorate,decoration={brace,mirror,amplitude=8pt}] (-1.5,6)--(-5.8,6) node [black,midway,yshift=.6cm] {$B'$};
\draw [decorate,decoration={brace,mirror,amplitude=8pt}] (9,6) --(1.5,6)node [black,midway,yshift=.6cm] {$\overline{B'}$};

\node  (S) at (-7.3,0) {$S$};
\draw [decorate,decoration={brace,mirror,amplitude=8pt}] (-6,5.8)--(-6,1.5) node [black,midway,xshift=-.6cm] {$B$};
\draw [decorate,decoration={brace,mirror,amplitude=8pt}] (-6,-1.5)--(-6,-9) node [black,midway,xshift=-.6cm] {$\overline{B}$};
\end{tikzpicture}
\end{center}
\caption{}\label{fig:TSS}
\end{figure}

\begin{cl}\label{claim:123} The following holds:
\begin{enumerate}
\item $B\cap\overline{B'}$ or $\overline{B}\cap B'$ is $X$\!-free,
\item $B\cap B'$ or $\overline{B}\cap \overline{B'}$ is $X$\!-free,
\item If $B\cap B'=\emptyset$, then $|T(B,B')|\geq 5$. 
\end{enumerate}
\end{cl}

\begin{smallproof}[of Claim~\ref{claim:123}]
To prove (i) assume that $B\cap\overline{B'}$ and $\overline{B}\cap B'$ are not $X$\!-free.
Then, by Claim~\ref{lem:T}, $T(B,\overline{B'})$ and $T(\overline{B},B')$ both are $X$\!-separators and since $|T(B,\overline{B'})|+|T(\overline{B},B')|=|S|+|S'|=8$, we have $|T(B,\overline{B'})|=|T(\overline{B},B')|=4$.
But $T(B,\overline{B'})\setminus\{y\}$ is also an $X$\!-separator because $y$ has no neighbor in $B\cap\overline{B'}$, a contradiction. \\
By the same arguments, $T(B,B')\setminus\{y\}$ is an $X$\!-separator of size $3$ if $B\cap B'$ and  $\overline{B}\cap \overline{B'}$  both are not $X$\!-free, and (ii) is shown. \\
To see (iii), assume that $B\cap B'=\emptyset$. Since $t$ and $t'$ must have neighbors in $B$ and $B'$, respectively, which can only be in $(S'\cap B)\setminus\{t'\}$ and $(S\cap B')\setminus\{t\}$, respectively, $T(B,B')$ has at least five vertices. 
\end{smallproof}

\begin{cl}\label{claim:TXdisj}
Let $T$ be a separator of $G$ with $|T|=3$. 
If $B$ and $B'$ are  an $S$-$X$\!-fragment and an $S'$\!-$X$\!-fragment, respectively, as defined before, then $\overline{B}\cap \overline{B'}$ is $X$\!-free.
Moreover, $T\cap X=\emptyset$.
\end{cl}

\begin{smallproof}[of Claim~\ref{claim:TXdisj}]
Assume that $\overline{B}\cap \overline{B'}$ is not $X$\!-free.
Thus, $B\cap B'$ is $X$\!-free by Claim~\ref{claim:123}~(ii) and $T(\overline{B},\overline{B'})$ is an $X$\!-separator by Claim~\ref{lem:T}~(i); therefore, $|T(\overline{B},\overline{B'})|\geq 4$.
One checks that $|\overline{B}\cap S'|\geq |S\cap B'|$ and $|\overline{B'}\cap S|\geq |S'\cap B|$ 
(it follows from $|S|=4\leq |T(\overline{B},\overline{B'})|=|S\setminus(S\cap B')|+|\overline{B}\cap S'|=|S|-|S\cap B'|+|\overline{B}\cap S'|$, the other inequality follows similarly). \\
Furthermore, $|T(B,B')|\leq 4$ (Claim~\ref{lem:T}~(ii)) and by Claim~\ref{claim:123}~(iii), $B\cap B'\neq \emptyset$, so that $\hat{T}=N_G(B\cap B')=T(B,B')\setminus\{y\}$ is a separator of size 3 in $G$.
By Claim~\ref{claim:SFeins}, $B\cap B'$ is a $\hat{T}$-fragment and its unique vertex $b$ is adjacent to the three vertices in $\hat{T}$.
Let $v$ be the unique vertex from $T(B,B')\setminus\{t,t',y\}$. 
The situation is sketched in Figure~\ref{fig:claim6}. 

\begin{figure}[h]
\begin{center}
\begin{tikzpicture}[scale=.3][line width=42.5mm]
\filldraw[fill=black!50!, fill opacity=.2,rounded corners=8pt, line width=.7pt, draw=black] (-9.0,1)--(-9.0,-1)--(6,-1)--(6,1)--cycle;
\filldraw[fill=black!50!, fill opacity=.2,rounded corners=8pt, line width=.7pt, draw=black] (-1,9)--(-1,-6)--(1,-6)--(1,9)--cycle;

\node  (y) [circle, draw=black,fill=black, inner sep=0pt, minimum width=5pt,label={[label distance=-.16cm]60:$y$} ] at (-.4,-.3) {};
\node  (x) [circle, draw=black,fill=black, inner sep=0pt, minimum width=5pt,label={[label distance=.02cm]182:$t$} ] at (-4,0) {};
\node  (xp) [circle, draw=black,fill=black, inner sep=0pt, minimum width=5pt,label={[label distance=.02cm]88:$t'$} ] at (0,4) {};
\node  (xpp) [circle, draw=black,fill=black, inner sep=0pt, minimum width=5pt,label={[label distance=-.08cm]35:$t''$} ] at (3,-3) {};

\node  (b) [circle, draw=black,fill=black, inner sep=0pt, minimum width=6pt,label={[label distance=-.08cm]90:$b$} ] at (-6,4) {};
\node  (v) [circle, draw=black,fill=white, inner sep=0pt, minimum width=6pt,label={[label distance=-.08cm]180:$v$} ] at (-1.1,.6) {};
\node  (c) [circle, draw=black,fill=white, inner sep=0pt, minimum width=6pt,label={[label distance=-.08cm]-90:$c$} ] at (-6,-4) {};
\node  (cs) [circle, draw=black,fill=white, inner sep=0pt, minimum width=5pt ] at (0,-4) {};

\draw[line width=.7pt] (y)--(x);
\draw[line width=.7pt] (y)--(xp);
\draw[line width=.7pt] (y)--(xpp);
\draw[line width=.7pt] (b)--(x);
\draw[line width=.7pt] (b)--(xp);
\draw[line width=.7pt] (b)--(v);
\draw[line width=.7pt,dashed] (c)--(x);
\draw[line width=.7pt,dashed] (c)--(cs);
\draw[line width=.7pt,dashed] (c)--(v);

\node (free) at (-5,8.4) {{\footnotesize $X$-free}};

\node  (Sp) at (0,7.3) {$S'$};
\draw [decorate,decoration={brace,mirror,amplitude=8pt}] (-1.5,9)--(-8.8,9) node [black,midway,yshift=.6cm] {$B'$};
\draw [decorate,decoration={brace,mirror,amplitude=8pt}] (6,9) --(1.5,9)node [black,midway,yshift=.6cm] {$\overline{B'}$};

\node  (S) at (-7.3,0) {$S$};
\draw [decorate,decoration={brace,mirror,amplitude=8pt}] (-9,8.8)--(-9,1.5) node [black,midway,xshift=-.6cm] {$B$};
\draw [decorate,decoration={brace,mirror,amplitude=8pt}] (-9,-1.5)--(-9,-6) node [black,midway,xshift=-.6cm] {$\overline{B}$};
\end{tikzpicture}
\end{center}
\caption{}\label{fig:claim6}
\end{figure}

If $v\in S\cap B'$, then $|T(\overline{B},B')|\geq 5$ (since $|\overline{B}\cap S'|\geq |S\cap B'|\geq 2$), which implies that 
 $|T(B,\overline{B'})|\leq 3$. Because $y\in T(B,\overline{B'})$ and $y$ has no neighbour in $B\cap \overline{B'}$ (remember that $N_G(y)=\{t,t',t''\}$), it follows that
$B\cap \overline{B'}$ is empty (otherwise  $T(B,\overline{B'})\setminus\{y\}$ was a separator of size at most 2).
Moreover,  $|B \cap S| = |B \cap S' | = 1$, and therefore $B\cap S' = \{t'\}$.\\
It follows that $B=\{b,t'\}$ and $t'\in X$, so that $t'$ has degree at least 4 and must be adjacent to at least one of the two neighbors of $b$ in $S$;
this is not possible as $N_G(b)=\hat{T}$ is an anticlique (Claim~\ref{claim:Tanti}).
Analogously, the assertion $v\in S'\cap B$ is contradictory. \\
It follows that $v\in S\cap S'$.
Thus, $|T(B,\overline{B'})|=|T(\overline{B},B')|=4$ using $|S|=|S'|=4$, where $y$ has no neighbors in $B\cap\overline{B'}$ and $\overline{B}\cap B'$, so that the latter two sets are $X$\!-free.
It follows that $X\cap B'=\{t\}$ and $t$ has degree at least 4. Since $t$ is non-adjacent to the two neighbors of $b$ in $S'$, it must have a neighbor in $B'$ distinct from $b$, implying that $\overline{B}\cap B'$ is non-empty and, consequently, consists of a single vertex $c$.
Since $t$ is not adjacent to the two neighbors of $c$ in $S'$, the only neighbors of $t$ are $b$, $c$, and $y$, a contradiction. \\
Therefore, $\overline{B}\cap \overline{B'}$ is $X$\!-free, and, in particular, $t''\notin X$. By symmetry, $t,t'\notin X$, so that $X$ is disjoint from $T$. 
\end{smallproof}

Let $B,B'$ as before and note that $\overline{B}\cap \overline{B'}\neq \emptyset$ is $X$\!-free (by definition and by Claim~\ref{claim:TXdisj}).
By symmetry we may assume that $\overline{B}\cap B'$ is $X$\!-free (see Claim~\ref{claim:123}~(i)), so that $\overline{B}\cap X \subseteq S'\cap \overline{B}$.
This implies that $T(\overline{B},\overline{B'})$ is not a separator in $G$ of size 3 (since $T(\overline{B},\overline{B'})$ is not $X$\!-free).
Thus, $|T(\overline{B},\overline{B'})|\geq 4$  and $|T(B,B')|\leq 4$ by Claim~\ref{lem:T}~(ii).
By Claim~\ref{claim:123}~(iii), $B\cap B'$ is non-empty, and, as $N_G(B\cap B')=T(B,B')\setminus\{y\}$ is a separator of size 3, we get by Claim~\ref{claim:SFeins} that $B\cap B'$ consists of a single vertex $b$ adjacent to all vertices in $T(B,B')\setminus\{y\}$, and, hence $b$ is adjacent to all vertices in $B'\cap S$; among them, there is at least one vertex from $B'\cap X$ (since $\overline{B}\cap B'$ and $B\cap B'$ are $X$\!-free).
This contradicts Claim~\ref{claim:TXdisj} that $N_G(b)$ must be $X$\!-free; and Lemma~\ref{lem:proofmainthm} is proved.
\end{proof}

\begin{proof}[of Theorem~\ref{main}~(i)]
If $k\in\{1,2,3\}$, then Theorem~\ref{main}~(i) follows immediately from (ii),
since a topological $X$\!-minor is an $X$\!-minor.
So let $k=4$. 
We construct a sequence of graphs $G_0=G, G_1,G_2,\dots $ with $|G_i|-1=|G_{i+1}|\geq |X|$, $X\subseteq V(G_i)$, and $\kappa_{G_i}(X)\geq 4$ for all $i$. 
We can assume that $G$ is connected, otherwise we take the unique component of $G$ containing vertices of $X$ as the graph $G_0$.
The graph $G_{i+1}$ is obtained from $G_i$ by contracting an $X$\!-legal edge $vy$ such that $\kappa_{G_i/vy}(X)\geq 4$. 
This edge $vy$ exists by Lemma~\ref{lem:proofmainthm}, unless $G_i$ is $4$-connected and we stop building the sequence with $G_i$. 
The last graph of the sequence is obtained from $G$ by a sequence of $X$\!-legal edge contractions, i.\,e.\ it is an $X$\!-minor of $G$, which is the desired 4-connected $X$\!-minor.
\end{proof}

\section{Proofs of Theorems \ref{f3} and \ref{S}}\label{P12}

In this section, we present the two missing proofs. 
Theorem~\ref{f3} is a straight consequence of the statements and Theorem~\ref{main}~(ii)
and will be presented next. 
Note that a minor of a graph $G$ does not contain a graph $U$ as a minor if already $G$ does not contain $U$ as a minor and that a minor of a planar graph  is also planar.

By defining $X$\!-spanning generalized cycles and paths, it is possible to adapt the proof idea of Theorem~\ref{f3} by using the result on $X$\!-rooted minors and prove a slightly different version of  Theorem~\ref{S}. 
We will give a glimpse of this, although the main part of this section focuses on the proof of Theorem~\ref{S}, which  needs more effort and uses the theory of Tutte paths. 

\begin{proof}[of Theorem~\ref{f3}]
Let $G$ be a graph and $X\subseteq V(G)$ with $\kappa_G(X)\ge 3$ and properties requested as in Theorem~\ref{f3}. 
By Theorem~\ref{main}~(ii), there is a $3$-connected topological $X$\!-minor $M$ of $G$.
If $x_1x_2\in E(G[X])$ is an edge that is not present in $M$, then we can add the edge $x_1x_2$ to $M$ and $M$ is still a $3$-connected topological $X$\!-minor of $G$.
Let $\varphi$ be an isomorphism  from a certain subdivision of $M$ into a subgraph  of $G$ such that all vertices of $M$ are fixed by $\varphi$. 
Applying the suitable Statement~\ref{Barnette}, ~\ref{Gao}, or~\ref{OtOz} on $M$, we obtain a spanning subgraph $H$ of $M$ containing all vertices from $X$. 
Using the isomorphism $\varphi$, a subdivision of $H$ can be found in $G$ which is $X$\!-spanning and has the properties in $G$ that $H$ has in $M$. 
\end{proof}

 Given a graph $G$ and $X\subseteq V(G)$, a subgraph $H$ of $G$ is an \emph{$X$\!-spanning generalized cycle} of $G$ if $H$ is the edge disjoint union of a cycle $C$ of $G$ and $|X|$ pairwise vertex disjoint paths $P[x_i,y_i]$ of $G$ connecting $x_i$ and $y_i$ (possibly $x_i=y_i$) such that  $X\cap V(P[x_i,y_i])=\{x_i\}$ and $V(C)\cap V(P[x_i,y_i])=\{y_i\}$ for $i=1,\dots,|X|$. An \emph{$X$\!-spanning  generalized path} $P$ of $G$ is defined similarly if in the previous definition the \emph{cycle $C$} is replaced with a \emph{path $P$} of $G$. Note that an $X$\!-spanning path or an $X$\!-spanning cycle is also an  $X$\!-spanning  generalized path or  an  $X$\!-spanning  generalized cycle, respectively, and we observe:

\begin{obs}\label{xspann-obs}
Let $X\subseteq V(G)$ for some graph $G$ and $M$ be an $X$\!-minor of $G$. 
If $M$ has an $X$\!-spanning  path  or an $X$\!-spanning  cycle as a subgraph, then $G$ contains an $X$\!-spanning generalized path  or an $X$\!-spanning generalized cycle, respectively. 
\end{obs}

\begin{proof}
Let $P$ be an $X$\!-spanning  path  of $M$ and $\mathcal{M}=(V_v)_{v\in V(M)}$ be an $M$-certificate. 
For each edge $uv\in E(P)$, there is an edge $e_{uv}\in E(G)$ between a vertex in $V_u$ and a vertex in $V_v$. 
For each $v\in V(P)$ we define a set $E_v$ of edges in $V_v$ as follows: If $v$ is an end vertex of $P$ or if, for $uv,vw\in E(P)$ with $u\neq w$, the end vertices of $e_{uv}$ and $e_{vw}$ in $V_v$ coincide, then $E_v=\emptyset$. 
Otherwise, the end vertices of $e_{uv}$ and $e_{vw}$ in $V_v$ can be connected by a path $Q$ in $G[V_v]$, since $G[V_v]$ is connected, and we put $E_v=E(Q)$. \\
We obtain a path $P'$ in $G$ with $E(P')=\{e_{uv}\mid uv\in E(P)\}\cup\bigcup_{v\in V(P)}E_v$, which has non-empty intersection with $V_v$ for all $v\in V(P)$. 
If $x\in X$ is not on $P'$, then there is a path $P_x$ in $V_x$ connecting $x$ to the subpath of $P'$ in $G[V_x]$, i.\,e.\ $X\cap V(P_x)=\{x\}$ and $|V(P')\cap V(P_x)|=1$. 
Eventually, $P'$ together with all paths $P_x$ for $x\in X\setminus V(P')$ forms an $X$\!-spanning generalized path of $G$.\\
Using the same arguments, the existence of an $X$\!-spanning generalized cycle of $G$ can be proved if $M$ contains an $X$\!-spanning cycle. 
\end{proof}

Using Theorem~\ref{main}~(i) and the previous Observation~\ref{xspann-obs}, Statements~\ref{Sanders} and~\ref{ToYu} can be immediately translated to locally spanning versions if the formulations ``\emph{$X$\!-spanning  path}'' and ``\emph{$(X\setminus Y)$-spanning cycle}'' in  Theorem~\ref{S}~(i) and~(ii) for the case $Y\subseteq X$
are replaced with ``\emph{$X$\!-spanning generalized path}'' and ``\emph{$(X\setminus Y)$-spanning generalized cycle}'', respectively. 
Theorem~\ref{S}~(i) and~(ii) do not follow directly from Theorem~\ref{main} since  Theorem~\ref{main}~(ii) is not true in case $k=4$ (see Observation~\ref{thm1-obs2}). 
We will use the theory of Tutte-paths in 2-connected plane graphs (see~\cite{sanders1997paths,thomas19944,thomas2005hamilton, thomassen1983theorem, tutte1956theorem}) instead of Theorem~\ref{main} to prove the strong  locally spanning versions, stated in Theorem~\ref{S}~(i) and~(ii), of  Statements~\ref{Sanders} and~\ref{ToYu}, respectively. 

Furthermore, we show  that Theorem~\ref{S} (iii) is a consequence of Statement~\ref{Ellingham} and Theorem~\ref{main}~(i); thereby the upper bound on the maximum degree of the desired tree increases by ``$+1$'' compared to the one of Statement~\ref{Ellingham} (observe again that Theorem~\ref{main}~(ii) does not hold in case $k=4$).

\begin{proof}[of Theorem~\ref{S}]
In the following proof of Theorem~\ref{S}, Observation~\ref{lem2-obs} obtained from Lemma~\ref{Lemma2}  is used several times.

\begin{obs}\label{lem2-obs}
Let $G$ be a  graph, $S\subset V(G)$ be a separator of $G$, and $F$ be an $X$\!-free $S$-fragment of $G$. Furthermore, let $G'$ be the graph obtained from  $G[\overline{F}\cup S]$ by adding all possible edges between vertices of $S$ (if not already present).\\
Then  $\kappa_{G'}(X)\ge \kappa_G(X)$ and $G'$ is planar if all following conditions hold: $G$ is planar, $|S|\le 3$, and $S$ is a minimal separator. 
\end{obs}

 Before we start to prove Theorem~\ref{S}~(i), we introduce the concept of bridges and Tutte paths~\cite{tutte1956theorem}, on which the proofs of Statements~\ref{Sanders} and~\ref{ToYu} are principally based.
We apply the widely used notation from~\cite{thomas2005hamilton} instead of the terminology from the original paper~\cite{tutte1956theorem} by Tutte. 
Therefore, let $G$ be a connected graph, $H$ be a subgraph of $G$, $V(G)\setminus V(H)\neq \emptyset$, and $K$ be a component of $G-V(H)$. If $N_G(K)\subseteq V(H)$ is the set
of neighbors of $K$ in $V(H)$, then the graph $B$  with  $V(B)=V(K)\cup N_G(K)$ and $E(B)=E(K)\cup \{uv\in E(G)\mid u\in V(K), v\in  V(H)\}$ is a non-trivial \emph{bridge} of $H$, where  $N_G(K)$ and
$V(K)$ are called the sets $T(B)$ of \emph{attachments} and $I(B)$ of  \emph{inner vertices} of $B$, respectively.
(A \emph{trivial bridge} is an edge of $G-E(H)$ whose two end vertices are contained in $H$.) 
Since we are interested in bridges containing a vertex of $X$ as an inner vertex, all references to bridges focus on non-trivial ones.

A path $P$ of $G$ on at least two vertices is a \emph{Tutte path} of $G$ if  each bridge of $P$ has at most three attachments. 
Let $H$ be a subgraph of $G$, then a path $P$ of $G$ on at least two vertices is an \emph{$H$\!-Tutte path} of $G$ if  each bridge of $P$ has at most three attachments and each bridge containing an
edge of $H$ has at most two attachments. 
A \emph{Tutte cycle} and an \emph{$H$\!-Tutte cycle} are defined the same way if in the previous definition the path $P$ is replaced by an cycle. 

In order to state Tutte's original result from~\cite{tutte1956theorem}, we assume that $G$ is a $2$-connected graph embedded into the plane.
The \emph{exterior cycle} of $G$ is the cycle $C_G$ bounding the
infinite face of $G$. 
Tutte proved that, for $y,z\in V(C_G)$ and $e\in E(C_G)$,  $G$ contains a $C_G$-Tutte path from $y$ to $z$ containing $e$.  Thomassen~\cite{thomassen1983theorem}
improved Tutte's result by removing the restriction on the location of $z$, and, eventually,  Sanders (\cite{sanders1997paths}) established the following Lemma~\ref{Sanders-path}:

\begin{lemma}[D.P. Sanders, 1997,~\cite{sanders1997paths}]\label{Sanders-path} 
If $G$ is a $2$-connected plane graph, $e\in E(C_G)$, and $y,z\in V(G)$, then $G$ has a  $C_G$-Tutte path from $y$ to $z$
containing $e$. \end{lemma}

The following lemma generalizes Tutte's result. We write  $G\cup H$ for the union of two graphs $G$ and $H$. 

\begin{lemma}[R. Thomas, X. Yu, 1994,~\cite{thomas19944}]\label{ToYu-2cycle-lemma}
If $G$ is a $2$-connected plane graph with outer cycle $C_G$, another facial cycle $D$, and $e\in E(C_G)$, then $G$ has an $(C_G\cup D)$-Tutte cycle $C$ such that $e\in E(C)$ and no $C$-bridge contains edges of both $C_G$ and $D$.
\end{lemma}

The next lemma describes where the vertices from some $X\subseteq V(G)$ of a graph $G$ are located relative to some Tutte path. 

\begin{lemma}\label{bridge} 
Let $G$ be a $2$-connected graph, $X\subseteq V(G)$, $\kappa_G (X)\ge 4$, and let $Q$ be an Tutte path of $G$. 
If $X$ is not a subset of $V(Q)$, then $X\subseteq V(B)$
for some bridge $B$ of $Q$ and, in this case, $Q$ contains at most $3$ vertices of $X$. \end{lemma}

\begin{smallproof}
Let $x\in X\setminus V(Q)$, then there is a bridge $B$ of $Q$ containing $x$ as an inner vertex and $B$ has at most three attachments on $Q$.
Assume there is a vertex $x'\in X\setminus V(B)$. Then the attachments $T(B)$ form an $X$\!-separator of $G$, contradicting $\kappa_G (X)\ge 4$. Hence, $X\subseteq V(B)$
and $|X\cap V(Q)|\le |V(B)\cap V(Q)|=|T(B)|\le 3$.
\end{smallproof}

\begin{proof}[of Theorem~\ref{S}~(i)]
 Suppose, to the contrary, that Theorem~\ref{S}~(i) does not hold and let $G$ be a counterexample  such that $|V(G)|$ is minimum.

 If $G$ is not $2$-connected, then, because $\kappa_G(X)\ge 4$,  $X\subseteq V(K)$ for a block $K$ of $G$.  Moreover, $\specialedges\subset E(K)$ and,  by Lemma~\ref{Lemma2}, $\kappa_K(X)\ge \kappa_G(X)\ge 4$. Thus, $K$
is a smaller counterexample than $G$,  a contradiction. 

 Assume that $G$ has a separator  $S=\{ u,v\} \subseteq V(G)$. Because $\kappa_G (X)\ge 4$, there is an $S$-fragment $F$, such that $X\subseteq F\cup S$.
Let $G_1$ be obtained from $G[F\cup S]$ by adding the edge $uv$ (if not already present).
By Observation~\ref{lem2-obs}, it follows $X\subseteq V(G_1)$, $\specialedges\subset E(G_1)$, and $\kappa_{G_1}(X)\ge 4$. 
Since  $G-F$ contains $S$, there is a path $Q$ of $G-F$ with ends $u$ and $v$. 
If $G_1$ has a path $P_1$ satisfying Theorem~\ref{S}~(i), then we can replace the edge $uv$ (if $uv \in E(P_1)\setminus E(G)$) by $Q$ and in both cases this will give us the required path $P$ in $G$. 
Therefore $G_1$ has no such path,
contradicting the minimality of $G$ as a counterexample.

 Hence, we may assume that $G$ is $3$-connected and consider two cases to complete the proof of Theorem~\ref{S}~(i).

 \emph{Case $1$}. $|\specialedges|=1$.\\ 
Let $Q$ be a Tutte path of $G$ connecting $x_1$ and $x_2$ such that $\specialedges\subset E(Q)$ (Lemma~\ref{Sanders-path}). 
If $X\subseteq V(Q)$, then $Q$ is the desired path $P$, contradicting the choice of $G$.
Otherwise,  it follows $|V(Q)\cap X|\leq 3$ by Lemma~\ref{bridge} and there is a bridge $B$ of $Q$ such that $X\subseteq V(B)$, $I(B)\cap X\neq \emptyset$.
Since $\specialedges$ consists of one edge $e$ from $E(G[X])$ and $x_1x_2\notin \specialedges$, we may assume that $e=x_1u$ for some vertex $u\in X$. 
 Hence, $T(B)=\{x_1,x_2,u\}$.\\
If $|V(Q)|\ge 4$ or $Q$ has a second bridge distinct from $B$, then the graph $G_1$  obtained from $G[I(B)\cup T(B)]$ by adding all possible edges between vertices of $T(B)$ (if not already present --- see Lemma \ref{Lemma2} with $T(B)$ as separator) fulfils $X\subseteq V(G_1)$, $e\in E(G_1)$, and $\kappa_{G_1}(X)\ge 4$. 
If $G_1$ has a path $P_1$ satisfying Theorem~\ref{S}~(i), then it contains the edge $e=x_1u$ and therefore misses the edges $x_1x_2$ and $x_2u$. 
The path $P_1$ is also a path in $G$ and is a required path. 
Therefore $G_1$ has no such path,
contradicting the minimality of $G$ as a counterexample.\\
Thus, $G=G_1$ if $x_1x_2\in E(G)$ or $G$ is obtained from $G_1$ by removing $x_1x_2$ otherwise. Moreover,  $Q$ is the path of length 2 on vertices $x_1,u,x_2$. 
Let $G'=G-x_1$ and assume that $G'$ is embedded in the plane.
Note that $G'$ is $2$-connected and therefore, there exists a face that contains the vertex $x_1$ and a facial cycle $C$ bounding this face. Then $u\in V(C)$ and let $e'$ be an edge of $C$ other than $ux_2$. 
By Lemma~\ref{Sanders-path}, there exists a $C$-Tutte path $R$ of $G'$ from $u$ to $x_2$ through the edge $e'$. 
If $X\setminus\{x_1\}\subseteq V(R)$, then the path obtained from $R$ by adding $x_1$ and $e=x_1u$ would be a path of $G$ containing $X$ and $\specialedges$, a contradiction. Note that $x_1x_2\notin E(R)$.\\
Otherwise, there is a bridge $B'$ of $R$ such that $I(B')\cap X\neq\emptyset$. If the component of $G-R$ containing $I(B')$ contains $x_1$, i.\,e.\ the bridge $B'$ is adjacent to $x_1$ if we put $x_1$ back, then $I(B')$ contains a vertex of the facial cycle $C$ and, therefore, the bridge $B'$ contains edges of $C$. 
Thus, $B'$ has at most two attachments and $T(B')\cup\{x_1\}$ is a 3-separator in $G$. 
By Lemma~\ref{bridge}, $X\setminus\{x_1\}\subseteq V(B')$ and therefore $T(B')=\{u,x_2\}$. 
The graph $G_2$  obtained from $G[V(B')\cup \{x_1\}]$ by adding all possible edges between vertices of $T(B')\cup\{x_1\}$ (if not already present --- see Lemma \ref{Lemma2} with $T(B')\cup\{x_1\}$ as separator) fulfils $X\subseteq V(G_2)$, $e\in E(G_2)$, and $\kappa_{G_2}(X)\ge 4$. Since $e'\neq ux_2$, $|V(G_2)|<|V(G)|$ and by the minimality of $G$, the graph $G_2$ contains an
$X$\!-spanning path $P_2$ connecting $x_1$ and $x_2$ with $e\in E(P_2)$. Moreovere, $x_1x_2\notin E(P_2)$. 
Because  $G_2$ is a subgraph of $G\cup\{x_1x_2\}$, it follows that $P_2$ is a desired path of $G$. \\
If the component of $G-R$ containing $I(B')$ does not contain $x_1$, then $T(B')$ is also a 3-separator in $G$. This separator separates a vertex from $X\cap I(B')$ from $x_1$, a contradiction to $\kappa_G (X)\ge 4$.

 \emph{Case $2$}. $\specialedges=\emptyset$.\\
 Choose an arbitrary edge $e=uv$  of $G$ such that $\{u,v\}\cap \{ x_1,x_2\}=\emptyset $. To see that $e$ exists,  assume that each edge of $G$ is incident with $x_1$ or with $x_2$. Then
 $G-\{x_1,x_2\}$ is edgeless, a contradiction to $\kappa_G(X)\ge 4$ and $|X|\ge 5$.\\
Now consider  a Tutte path $Q$ from $x_1$ to $x_2$ through $e$. Since  $X\subseteq V(Q)$ contradicts the choice of $G$, there exists a bridge $B$ of $Q$ such that $X\subseteq V(B)$. It follows
$X\cap I(B)\neq \emptyset$ and  $x_1,x_2\in T(B)$. Since $|V(Q)|\ge 4$, the graph $G_1$ obtained from $G[I(B)\cup T(B)]$ by adding all possible edges between vertices of $T(B)$ (if not already present), 
fulfils $X\subseteq V(G_1)$ and $\kappa_{G_1}(X)\ge 4$. 
If $G_1$ has a path $P_1$ satisfying Theorem~\ref{S}~(i), then it misses the edge $x_1x_2$. 
We can replace in $P_1$ the other two edges of $G_1[T(B)]$ (if they exist) by subpaths of $Q$ and this will give us a required path $P$ in $G$. 
Therefore $G_1$ has no such path,
contradicting the minimality of $G$ as a counterexample.
\end{proof}

\begin{proof}[of Theorem~\ref{S}~(ii)]
Suppose, to the contrary, that Theorem~\ref{S}~(ii) does not hold and let $G$ be a counterexample  such that $|V(G)|$ is minimum. 
Moreover, we assume that the counterexample $G$ with $X,Y$ maximizes $|Y|+|Y\cap X|$. 

If $G$ is not $2$-connected, then,  as in the proof of Theorem~\ref{S}~(i), there is a block $K$
of $G$ with $X\subseteq V(K)$ and $\kappa_K(X)\ge 4$. Thus, $K-Y$
contains an $(X\setminus Y)$-spanning cycle by the minimality of $G$.

\emph{Case $A$:}
First we consider the case $Y=\{y_1,y_2\} \subseteq X$. \\
Assume that $G$ is embedded in the plane such that $y_1$ is incident
with the outer face and consider $G-\{y_1,y_2\}$.
Since $|X|\ge 5$ (because $\kappa_G(X)\geq 4$) and $\kappa_{(G-\{y_1,y_2\})}(X\setminus\{y_1,y_2\})\geq 2$,
there is a block $H$ containing $X\setminus\{y_1,y_2\}$ of $G-\{y_1,y_2\}$. 

Assume there is a component $K$ of $G-(\{y_1,y_2\}\cup V(H))$
and let $N_G(K)$ be
the neighbors of $K$ in $G$.
Because $H$, as a block of $G-\{y_1,y_2\}$, is a maximal $2$-connected subgraph,  it follows  $|N_G(K)\cap V(H)|\leq 1$. Obviously, $N_G(K)\setminus V(H)\subseteq \{y_1,y_2\}$  and, therefore, $|N_G(K)|\le 3$.\\
Consider the  graph $G_1$ obtained from $G$ by removing $V(K)$ and adding all
edges between the vertices of  $N_G(K)$ (if not already present). Then $G_1$ is planar since $|N_G(K)|\leq 3$ and, furthermore, $\kappa_{G_1}(X)\ge 4$ (see Observation~\ref{lem2-obs}).
By the choice of $G$, there is a cycle $C$ of $G_1$ containing all vertices of
$X$ except $y_1$ and $y_2$. Evidently, $C$ misses all new edges between
the vertices of $N_G(K)$, thus, $C$ is also a  cycle of $G$, a contradiction.
We conclude that $H=G-\{y_1,y_2\}$. 

 For $i=1,2$, there are (not necessarily distinct) faces $\alpha_i$ of $H$ containing the vertex $y_i$ in $G$ and let $C_i$ be the facial cycle of $\alpha_i$ in $H$.
Because of the choice of the embedding of $G$, $\alpha_1$ is the outer face of $H$, thus, $C_H=C_1$.
We follow the proof in~\cite{thomas19944}.

 \emph{Case $1$}: $C_1=C_2$.\\
If $\alpha_1\ne\alpha_2$, then $H=C_1$ and $C_1$ is the desired cycle.
Otherwise, the vertices of $V(C_1)$ can be numbered with $v_1,v_2,\dots,v_k$ according to their cyclic order in a such way that $y_2$ is not adjacent to vertices $v_2,v_3,\dots, v_{l-1}$ and $y_1$ is not adjacent to vertices $v_{l+1},v_{l+2},\dots, v_k$ for some integer $l$ with $3\leq l\leq k-1$ (note that $y_1$ and $y_2$ have degree at least $4$ in $G$).
We apply Lemma~\ref{Sanders-path} and consider a $C_1$-Tutte path $Q$ of $H$ from $v_1$ to $v_2$ containing $v_lv_{l+1}$ which can be joined by $v_1v_2$ to a cycle. \\
Since $G$ is a counterexample, there is $x\in X\setminus (V(Q)\cup \{y_1,y_2\})$ and
a bridge $B$ of $Q$ in $H$ containing $x$ as an inner vertex.\\
If $I(B)\cap V(C_1)=\emptyset$, then $N_G(y_1)\cap V(B)\subseteq T(B)$ and
$T(B)$ separates $x$ from $y_1$ in $G$, contradicting $\kappa_G (X)\ge 4$.\\
Otherwise, there is  $v\in I(B)\cap V(C_1)$. Then the edge $uv$, where $u$ is a neighbor of $v$ in $C_1$, belongs to $B$. Especially, $u\in V(B)$ and $B$ has exactly two  attachments  $s$ and $t$ in $V(Q)$ and $s,t\in V(C_1)$. Thus, the subpath $P$ of $C_1$ from $s$ to $t$ containing $v$ is a path of $B$.
Because of planarity, the bridge cannot intersect $C_1$ outside the subpath $P$ and we obtain $(I(B)\cap V(C_1))\setminus V(P)= \emptyset$.

Furthermore, $v_1,v_l\notin I(B)$ and $(I(B)\cap V(C_1))\cap N_G(y_i)=\emptyset$ for one $i\in\{1,2\}$.
But then $N_G(y_i)\cap V(B)\subseteq T(B)$ and $T(B)\cup\{y_{3-i}\}$ separates $x$ from $y_i$, contradicting $\kappa_G (X)\geq 4$.

 \emph{Case $2$}: $C_1\ne C_2$.\\
By Lemma~\ref{ToYu-2cycle-lemma}, there is an
$(C_1\cup C_2)$-Tutte cycle $C$. \\
 Since $G$ is a counterexample, there is $x\in X\setminus (V(C)\cup\{y_1,y_2\})$ and
a bridge $B$ of $C$ containing $x$ as an inner vertex and, by Lemma~\ref{ToYu-2cycle-lemma},  not simultaneously edges from both cycles $C_1$ and $C_2$.
Hence, $I(B)\cap V(C_1)=\emptyset$ or $I(B)\cap V(C_2)=\emptyset$, and in both
cases $T(B)$ separates $x$ from $y_1$ or $y_2$ in $G$, contradicting $\kappa_G
(X)\ge 4$. 

\emph{Case $B$:}
Now we consider the case that $Y=\{y_1,y_2\}$ and $|Y \cap X|\le 1$. \\
Among all possible choices for the counterexample $G$, 
choose the one maximizing $d_G(y_1)+d_G(y_2)$. 
Note that the maximum exists because $|V(G)|$ is already determined. 

Let $G$ be embedded into the plane. If $y\in \{y_1,y_2\}$ is incident with a face of $G$ which has a vertex $z$  at its boundary that is not a neighbour of $y$, then $G+yz$ is still planar. It follows $\kappa_{(G+yz)}(X)\ge 4$ and, because of the choice of the counterexample, $(G+yz)-\{y_1,y_2\}=G-\{y_1,y_2\}$ has a desired cycle, a contradiction.  Hence, if $y\in \{y_1,y_2\}$ is incident with a face  of $G$, then all vertices at its boundary are neighbours of $y$.
Thus, since $G$ is simple, $|N_G(y)|\ge 3$ for $y\in \{y_1,y_2\}$.

Next we show that $G$ is $3$-connected. 
Let $S=\{u,v\}$ be a separator of $G$. 
Because $\kappa_G (X)\ge 4$, there is an $S$-fragment $F$ such that $X\subseteq F\cup S$.
Let $G_1$ be obtained from $G[F\cup S]$ by adding the edge $uv$ (if not already present).
By Observation~\ref{lem2-obs}, it follows $X\subseteq V(G_1)$ and $\kappa_{G_1}(X)\ge 4$. 
If $uv\in E(G)$, then $G_1-Y$ contains an $(X\setminus Y)$-spanning cycle, which is also a desired cycle in $G$, a contradiction. 
Suppose first that $\{y_1,y_2\}\subset V(G_1)$. 
As before $G_1-Y$ contains an $(X\setminus Y)$-spanning cycle $C$, which would be an $(X\setminus Y)$-spanning cycle of $G$ if $uv\notin E(C)$. 
It follows that  $uv\in E(C)$ and thus,  $u,v\notin Y$. 
There is a path $Q$ of $G-F$ with ends $u$ and $v$ that does not contain any vertex from $Y$. 
We can replace the edge $uv$ in $C$ by $Q$ and obtain an $(X\setminus Y)$-spanning cycle of $G-Y$, a contradiction.

What follows is that, say, $y_1\notin V(G_1)$. Since $|N_G(y_1)|\ge 3$, we have $|V(G_1)|+1<|V(G)|$. 
Let $G_2$ be obtained from $G[F\cup S]$ by adding a path on vertices $u,z,v$ between $u$ and $v$. Remember that the edge $uv$ is not present in $G$.
It is straightforward to check that  Lemma~\ref{Lemma2} also holds if paths (instead of edges) between two vertices are inserted (if there is no edge between those two vertices). 
Thus  $\kappa_{G_2}(X)\ge\kappa_{G}(X)\ge 4$, $G_2$ is planar, $X\subseteq V(G_2)$, and $|V(G_2)|=|V(G_1)|+1<|V(G)|$.
Let $Y_2=(Y\cap V(G_1))\cup\{z\}$.
Then  $G_2-Y_2$ contains an $(X\setminus Y_2)$-spanning cycle $C$, which is an $(X\setminus Y)$-spanning cycle of $G-Y$ (because $V(C)\subseteq F\cup S$), a contradiction. 
Hence, we may assume that $G$ is $3$-connected.

If $\kappa_G(Z)\ge 4$ for $Z=X\cup \{y\}$ where $y\in \{y_1,y_2\}\setminus X$, then $G-Y$ has a $(Z\setminus Y)$-spanning cycle by the initial choice of $G$, 
a contradiction because $Z\setminus \{y_1,y_2\}=X\setminus \{y_1,y_2\}$. 
Hence, $\kappa_G(Z)\le 3$ and, because $\kappa_G(X)\ge 4$, it follows that
for $y\in \{y_1,y_2\}\setminus X$, there is as $3$-separator $S$ separating $y$ from $X$. 

Assume from now on that $y_1\notin X$. 
There is a $3$-separator $S=\{u,v,w\}$ of $G$ and a minimal $S$-fragment $F$, such that $X\subseteq F\cup S$ and $y_1\in \overline{F}$. 
Let $G_1$ be obtained from $G[F\cup S]$ by adding a new  vertex (also named) $y_1$ and three edges $uy_1$, $vy_1$, and $wy_1$.
Let $Y_1=Y\cap V(G_1)$. If $G_1-Y_1$ contains an $(X\setminus Y_1)$-spanning cycle $C$, then $C$ is a cycle of $G[F\cup S]$ and, therefore,  an $(X\setminus Y)$-spanning cycle of $G-Y$,  a contradiction. 
Hence,  $G_1$ does not have such a cycle and is a smaller counterexample unless $G_1=G$.

It follows that $G[F\cup S] = G-\{y_1\}$. 
We have argued in the beginning of \emph{Case}~$B$ that for an embedding of $G$, all vertices of the facial cycle of the face of $G-\{y_1\}$ that is incident with $y_1$ are adjacent to $y_1$. 
Thus, $N_G(y_1)=S$ and $G[S]$ is a cycle. 
Let $G'=G[F\cup S]$. By Observation~\ref{lem2-obs}, it follows $G'$ is planar, $X\subseteq V(G')$ and $\kappa_{G'}(X)\ge 4$. 
Hence, $G'-\{y_2\}$ contains an $(X\setminus Y)$-spanning cycle, which is  an $(X\setminus Y)$-spanning cycle in $G-Y$, a contradiction. 
Hence, $|Y|<2$ and this is the last case to be considered in order to complete the proof of Theorem~\ref{S}~(ii). 

\emph{Case $C$:} $|Y|<2$.\\
If there is a vertex $z\in V(G)\setminus(X\cup Y)$, then let $Y'=Y\cup\{z\}$. 
By the choice of $G$, $G-Y'$ contains an $(X\setminus Y')$-spanning cycle, which is an $(X\setminus Y)$-spanning cycle in $G-Y$, too.\\
Hence, $V(G)=X\cup Y$. 
If $Y\subseteq X$, then $G$ is $4$-connected and contains an $(X\setminus Y)$-spanning cycle by an immediate corollary of Statement~\ref{Sanders} (since $|Y|\in\{0,1\}$). 

It remains to consider the case $Y=\{y_1\}$ and $y_1\notin X$. 
By the same arguments as in \emph{Case}~$B$ (omitting $y_2$ everywhere), it follows that 
$G[X] = G-\{y_1\}$ and for $S:=N_G(y_1)$, it is $|S|=3$ and $G[S]$ is a cycle. 
Hence, $G[X]$ is 4-connected and, therefore, contains a hamiltonian cycle, which is  an $(X\setminus Y)$-spanning cycle of $G$.
\end{proof}

\begin{proof}[of Theorem~\ref{S}~(iii)]
 Note that any minor of $G$ is also embeddable on a surface of Euler characteristic~$\chi$. 
 Using Theorem~\ref{main}~(i) and Statement~\ref{Ellingham}, let $M$ be a $4$-connected $X$\!-minor of $G$, $\mathcal{M}=(V_v)_{v\in V(M)}$ be an $M$-certificate of $G$, and $T$ be a spanning tree of $M$ of maximum degree at most $\lceil \frac{10-\chi }{4}\rceil$.\\
 For each edge $e=uv\in E(T)$, let $e'\in E(G)$ be an arbitrary edge between a vertex in $V_u$ and a vertex in $V_v$. Furthermore, set $V'_v=V_v\cap (\bigcup_{e\in E(T)}V(e'))$ for $v\in V(T)$. 
 Moreover, for $v\in V(T)$ and $w\in V'_v$, let $f(w)=|\{ e\in E(T)\mid w$ is incident with $e'\}|$. \\
 Since $\sum_{w\in V'_v}f(w)=d_T(v)$, it follows $1\le f(w)\le d_T(v)-|V'_v|+1$ for all $w\in V'_v$.
 Since $G[V_v]$ is connected for $v\in V(T)$, the following Observation~\ref{tree} can be seen readily by induction on $|V'_v|$.

 \begin{obs}\label{tree}
 For $v\in V(T)$, $G[V_v]$ contains a $V'_v$-spanning tree $T_{V'_v}$ such that, for all  $w\in V(T_{V'_v})$,  $d_{T_{V'_v}}(w)\le |V'_v|-1$ if $w\in V'_v$ and  $d_{T_{V'_v}}(w)\le |V'_v|$, otherwise.
 \end{obs}

  Let $T^*$ be the tree of $G$ with 
\begin{align*}
V(T^*)&=\bigcup_{v\in V(T)}V(T_{V'_v}) \text{ and} \\
E(T^*)&=(\bigcup_{v\in V(T)}E(T_{V'_v}))~\cup ~\{e'~|~e\in E(T)\}.
\end{align*}
 Since $f(w)\le d_{T}(v)-|V'_v|+1$ and $d_{T_{V'_v}}(w)\le |V'_v|-1$, it follows 
 $d_{T^*}(w)=f(w)+d_{T_{V'_v}}(w)\le d_{T}(v)$ for $w\in V'_v$ and $v\in V(T)$.
 If $w\in (V_v\setminus V'_v)\cap V(T^*)$ for some $v\in V(T)$, then 
 $d_{T^*}(w)=d_{T_{V'_v}}(w)\le  |V'_v|\le \sum_{u\in V'_v}f(u)= d_T(v)$.
All together, the maximum degree of $T^*$ is at most $\lceil \frac{10-\chi }{4}\rceil$.

  Clearly, $X\subseteq \bigcup_{v\in V(T)}V_v$ and $|X\cap V_v|\le 1$ for $v\in V(T)$.
For every  $v\in V(T)$ and $x\in X\cap (V_v\setminus V(T_{V'_v})) $, let  $P$ be a path of $G[V_v]$ connecting $x$ with a vertex $y$ of $T_{V'_v}$ such that $V(P)\cap V(T_{V'_v})=\{ y\}$ and add $P$ to $T^*$. The resulting graph is the desired $X$\!-spanning $(\lceil \frac{10-\chi }{4}\rceil+1)$-tree of $G$. 
\end{proof}
\end{proof}

\subsubsection*{Acknowledgement}
The authors would like to thank the two referees, whose extensive and constructive suggestions helped to improve the quality of this article.

\printbibliography

\end{document}